\documentclass[twoside,leqno,10pt, A4]{amsart}
\usepackage{amsfonts}
\usepackage{amsmath}
\usepackage{amscd}
\usepackage{amssymb}
\usepackage{amsthm}
\usepackage{amsrefs}
\usepackage{latexsym}
\usepackage{mathrsfs}
\usepackage{bbm}
\usepackage{amscd}
\usepackage{amssymb}
\usepackage{amsthm}
\usepackage{amsrefs}
\usepackage{latexsym}
\usepackage{mathrsfs}
\usepackage{bbm}
\usepackage{enumerate}
\usepackage{graphicx}
\usepackage{color}
\begin{document}

\newtheorem{theorem}[subsection]{Theorem}
\newtheorem{proposition}[subsection]{Proposition}
\newtheorem{lemma}[subsection]{Lemma}
\newtheorem{corollary}[subsection]{Corollary}
\newtheorem{conjecture}[subsection]{Conjecture}
\newtheorem{prop}[subsection]{Proposition}
\numberwithin{equation}{section}
\newcommand{\mr}{\ensuremath{\mathbb R}}
\newcommand{\mc}{\ensuremath{\mathbb C}}
\newcommand{\dif}{\mathrm{d}}
\newcommand{\intz}{\mathbb{Z}}
\newcommand{\ratq}{\mathbb{Q}}
\newcommand{\natn}{\mathbb{N}}
\newcommand{\comc}{\mathbb{C}}
\newcommand{\rear}{\mathbb{R}}
\newcommand{\prip}{\mathbb{P}}
\newcommand{\uph}{\mathbb{H}}
\newcommand{\fief}{\mathbb{F}}
\newcommand{\majorarc}{\mathfrak{M}}
\newcommand{\minorarc}{\mathfrak{m}}
\newcommand{\sings}{\mathfrak{S}}
\newcommand{\fA}{\ensuremath{\mathfrak A}}
\newcommand{\mn}{\ensuremath{\mathbb N}}
\newcommand{\mq}{\ensuremath{\mathbb Q}}
\newcommand{\half}{\tfrac{1}{2}}
\newcommand{\f}{f\times \chi}
\newcommand{\summ}{\mathop{{\sum}^{\star}}}
\newcommand{\chiq}{\chi \bmod q}
\newcommand{\chidb}{\chi \bmod db}
\newcommand{\chid}{\chi \bmod d}
\newcommand{\sym}{\text{sym}^2}
\newcommand{\hhalf}{\tfrac{1}{2}}
\newcommand{\sumstar}{\sideset{}{^*}\sum}
\newcommand{\sumprime}{\sideset{}{'}\sum}
\newcommand{\sumprimeprime}{\sideset{}{''}\sum}
\newcommand{\sumflat}{\sideset{}{^\flat}\sum}
\newcommand{\shortmod}{\ensuremath{\negthickspace \negthickspace \negthickspace \pmod}}
\newcommand{\V}{V\left(\frac{nm}{q^2}\right)}
\newcommand{\sumi}{\mathop{{\sum}^{\dagger}}}
\newcommand{\mz}{\ensuremath{\mathbb Z}}
\newcommand{\leg}[2]{\left(\frac{#1}{#2}\right)}
\newcommand{\muK}{\mu_{\omega}}
\newcommand{\thalf}{\tfrac12}
\newcommand{\lp}{\left(}
\newcommand{\rp}{\right)}
\newcommand{\Lam}{\Lambda_{[i]}}
\newcommand{\lam}{\lambda}

\theoremstyle{plain}
\newtheorem{conj}{Conjecture}
\newtheorem{remark}[subsection]{Remark}

\makeatletter
\def\widebreve{\mathpalette\wide@breve}
\def\wide@breve#1#2{\sbox\z@{$#1#2$}%
     \mathop{\vbox{\m@th\ialign{##\crcr
\kern0.08em\brevefill#1{0.8\wd\z@}\crcr\noalign{\nointerlineskip}%
                    $\hss#1#2\hss$\crcr}}}\limits}
\def\brevefill#1#2{$\m@th\sbox\tw@{$#1($}%
  \hss\resizebox{#2}{\wd\tw@}{\rotatebox[origin=c]{90}{\upshape(}}\hss$}
\makeatletter

\title[The fourth moment of central values of quadratic Hecke $L$-functions in the Gaussian field]{The fourth moment of central values of quadratic Hecke $L$-functions in the Gaussian field}

\author{Peng Gao}
\address{School of Mathematical Sciences, Beihang University, Beijing 100191, P. R. China}
\email{penggao@buaa.edu.cn}
\begin{abstract}
 We obtain an asymptotic formula for the fourth moment of central values of a family of quadratic Hecke $L$-functions in the Gaussian field under the generalized Riemann hypothesis (GRH). We also establish lower bounds unconditionally and upper bounds under GRH for higher moments of the same family.
\end{abstract}

\maketitle



\section{Introduction}
\label{sec 1}

  Moments of central values of families of $L$-functions have been intensively studied in the literature in order to understand important arithmetic information they carry.  Although much progress has been made towards establishing asymptotic formulas for the first few moments for various families of $L$-functions, little is known for the higher moments. In connection with random matrix theory, conjectures on the order of magnitude for the moments were made by J. P. Keating and N. C. Snaith in \cite{Keating-Snaith02}. A simple and powerful method towards establishing lower bounds of these
conjectured results was developed by Z. Rudnick and K. Soundararajan in \cite{R&Sound} and \cite{R&Sound1}.

  For the family of quadratic Dirichlet $L$-functions, the above method of Rudnick and Soundararajan allows them to show in \cite{R&Sound1} that for every even natural number $k$,
\begin{align}
\label{lower}
 \sum_{\substack{ d \in \mathcal{D} \\ |d| \leq X}}L(\tfrac{1}{2},\chi_d)^k \gg_k X(\log X)^{\frac{k(k+1)}{2}},
\end{align}
 where $\chi_d = \left(\frac{d}{\cdot} \right)$ is the Kronecker symbol and $\mathcal{D}$  denotes the set of fundamental discriminants.

 In the other direction, Soundararajan \cite{Sound01} proved that, assuming the generalized Riemann hypothesis (GRH), for all real $k \geq 0$,
 \begin{align}
\label{upper}
\sum_{\substack{ d \in \mathcal{D} \\ |d| \leq X}}L(\tfrac{1}{2},\chi_d)^k\ll_{k,\varepsilon} X(\log X)^{\frac{k(k+1)}{2}+\varepsilon}.
\end{align}
   The above result was further sharpened by A. J. Harper in \cite{Harper} to remove the $\varepsilon$ power.

 Besides the lower and upper bounds for the moments given in \eqref{lower} and \eqref{upper}, asymptotic formulas are known for integers $1 \leq k \leq 4$ in the form
\begin{align}
\label{Lasymp}
  \sum_{\substack{0<d\leq X \\ (d,2) =1 \\ \text{$d$ square-free} }}L(\tfrac{1}{2},\chi_{8d})^k   = XP_{\frac{k(k+1)}{2}}(\log X) + E_k(X),
\end{align}
  where $P_n(x)$ is an explicit linear polynomials of degree $n$ and $E_k(X)$ is the error term.  Here we note that for odd, square-free $d>0$, the character  $\chi_{8d}$ is primitive modulo $8d$ satisfying $\chi_{8d}(-1) = 1$. Note also that we choose to present the asymptotic formulas for a family which is preferred by most recent studies due to technical reasons instead of the family appearing in \eqref{lower} and \eqref{upper}.

 Evaluation of the first two moments ($k=1,2$) in \eqref{Lasymp} was initiated by M. Jutila in \cite{Jutila}. The error terms in Jutila's results were subsequently improved in \cites{ViTa, DoHo, Young1} for the first moment and in \cites{sound1, Sono} for the second moment. For smoothed first moment, a result of D. Goldfeld and J. Hoffstein in \cite{DoHo} implies that one can take $E_1(x)=O(X^{1/2 + \varepsilon})$. An error term of the same size was later obtained by M. P. Young in \cite{Young1} using a recursive approach. The same approach was then adapted by K. Sono \cite{Sono} to show that one can take $E_2(x)=O(X^{1/2 + \varepsilon})$ for smoothed second moment. The sizes of these error terms then make the expressions given in \eqref{Lasymp} in agreement with a conjecture made by J. B. Conrey, D. W. Farmer, J. P. Keating, M. O. Rubinstein and N. C. Snaith in \cite{CFKRS} on asymptotic behaviours of these moments of the family of quadratic Dirichlet $L$-functions.

  The third moment given in \eqref{Lasymp} was originally obtained by Soundararajan in \cite{sound1}. The error term in Soundararajan's result was later improved by Young in \cite{Young2} to be $E_3(X)=O(X^{\frac{3}{4}+\varepsilon})$ for the smoothed version.  Under GRH, Q. Shen \cite{Shen} established an asymptotic formula for the fourth moment given in \eqref{Lasymp} with some savings on the powers of $\log X$ in the error term. The result of Shen is achieved by combining lower and upper bounds for the fourth moment, with a lower bound obtained unconditionally using the method of Rudnick and Soundararajan in \cites{R&Sound, R&Sound1} (note that such a lower bound is stated in \cites{R&Sound1} without proof). An upper bound for the fourth moment is obtained under GRH by making use of upper bounds for the shifted moments of quadratic Dirichlet $L$-functions. This approach originated from earlier work of Soundararajan \cite{Sound01} as well as Soundararajan and Young \cite{S&Y}, especially the treatment in \cite{S&Y} on the second moment of quadratic twists of modular $L$-functions.

 The result of Soundararajan and Young in \cite{S&Y} supplies another example on moments of quadratic families of $L$-functions. In \cite{Gao1}, the author considered the analogue case of evaluating moments of central values of a family of quadratic Hecke $L$-functions in the Gaussian field to obtain the smoothed first three moments of such family with power saving error terms.

  To describe the family studied in \cite{Gao1}, we let $K=\mq(i)$ be the Gaussian field throughout the paper and $\mathcal{O}_K=\mz[i]$ be the ring of integers of $K$. We  denote $\chi$ for a Hecke character of $K$ and we recall that a Hecke character $\chi$ of $K$ is said to be of trivial infinite type if its component at infinite places of $K$ is trivial.  We  write $L(s,\chi)$ for the $L$-function associated to $\chi$ and $\zeta_{K}(s)$ for the Dedekind zeta function of $K$. In the rest of the paper, we reserve the expression $\chi_c$ for the the quadratic residue symbol $\leg {c}{\cdot}$ defined in Section \ref{sec2.4}. We also denote $N(n)$ for the norm of any $n \in K$.

  We then consider the family of $L$-functions:
\begin{align}
\label{Lfamily}
 \mathcal F = \big\{ L(\thalf, \chi_{(1+i)^5d}) : d \in \mathcal{O}_K \hspace{0.05in} \text{odd and square-free} \big\}.
\end{align}
  Here we say that any $d \in \mathcal{O}_K$ is odd if $(d,2)=1$ and $d$ is square-free if the ideal $(d)$ is not divisible by the square of any prime ideal. Note also that (see Section \ref{sec2.4} below) the symbol $\chi_{(1+i)^5d}$ defines a primitive quadratic Hecke character modulo $(1+i)^5d$ of trivial infinite type when $d \in \mathcal{O}_K$ is odd and square-free.

  The first three moments of the above family $\mathcal F$ is studied in \cite{Gao1}. In this paper, our main purpose is to establish the fourth moment of the same family under GRH. In this process, we shall need to make use of results on upper bounds for shifted moments of the family under GRH. Therefore, we begin by considering lower and upper bounds for higher moments of this family. Unconditionally, we have the following result concerning the lower bounds.
\begin{theorem}
\label{theo:lowerbound}
 For every even natural number $k$, we have
\begin{align}
\label{lowerbound}
 \sumstar_{\substack{(d,2)=1 \\ N(d) \leq X}} L(\thalf, \chi_{(1+i)^5d})^k  \gg_k  X(\log X)^{k(k+1)/2}.
\end{align}
 Here the ``$*$'' on the sum over $d$ means that the sum is restricted to square-free elements $d$ in $\mathcal{O}_K$.
\end{theorem}

  The proof of Theorem \ref{theo:lowerbound} is given in Section \ref{sec: pfthm1} and it follows the arguments in the proof of \cite[Theorem 2]{R&Sound1} for the case of Dirichlet $L$-functions. Also similar to the remarks given below \cite[Theorem 2]{R&Sound1}, the proof of Theorem \ref{theo:lowerbound} can be applied to give lower bounds for the moments for all rational numbers $k$, provided that we replace $L(\tfrac 12,\chi_d)^k$ by $|L(\tfrac 12,\chi_d)|^k$.

  Before we state a corresponding result on the upper bounds, we need to introduce some notations.  Let $x\in \mathbb{R}$ such that $x \geq 10$ and  $z \in \mathbb{C}$, we define
\begin{align}
\label{Ldef}
\mathcal{L}(z,x)=\left\{
 \begin{array}
  [c]{ll}
  \operatorname{log}\operatorname{log}x & |z|\leq(\operatorname{log}x)^{-1},\\
  -\operatorname{log}|z|& (\operatorname{log}x)^{-1}< |z|\leq 1,\\
  0&|z| >1.
 \end{array}
 \right.
\end{align}
 We further define for $z_1, z_2 \in \mathbb{C}$,
\begin{align}
\label{Mdef}
\mathcal{M}(z_1,z_2,x) =&\frac{1}{2} \left(\mathcal{L}(z_1,x)+\mathcal{L}(z_2,x) \right), \\
\label{Vdef}
\mathcal{V}  (z_1,z_2,x)=& \frac{1}{2}\left(\mathcal{L}(2z_1,x)+\mathcal{L}(2z_2,x)+\mathcal{L}(2 \Re (z_1),x)+\mathcal{L}(2\Re (z_2),x)+2\mathcal{L}(z_1+z_2,x)+2\mathcal{L}(z_1+\overline{z_2},x)\right).
\end{align}

   Then we have the following result on upper bounds for shifted moments of the family $\mathcal F$  given in \eqref{Lfamily} under GRH.
\begin{theorem}
\label{upperbound}
Assume GRH for $\zeta_K(s)$ and $L(s, \chi_{(1+i)^5d})$ for all odd, square-free $d$.  Let $X$ be large and let $z_1,z_2 \in \mathbb{C}$ with $0 \leq \Re (z_1), \Re (z_2) \leq \frac{1}{\log X}$, and $|\Im (z_1)|, |\Im (z_2)| \leq X$. Then for any positive real number $k$ and any $\varepsilon>0$, we have
 \begin{align*}
 \sumstar_{\substack{(d,2)=1 \\ N(d) \leq X}}|L(\tfrac{1}{2}+z_1,\chi_{(1+i)^5d})
 L(\tfrac{1}{2}+z_2,\chi_{(1+i)^5d})|^{k}
  \ll_{k,\varepsilon}X(\operatorname{log}X)^\varepsilon
  \operatorname{exp}\left(k \mathcal{M}(z_1,z_2,X)+\frac{k^2}{2}\mathcal{V}(z_1,z_2,X)\right).
 \end{align*}
\end{theorem}

   The proof of  Theorem \ref{upperbound} is given in Section \ref{sec:upper bd} and our proof follows closely the approaches in \cites{Sound01,S&Y, Shen}. We note that similar  results to Theorem \ref{upperbound} were obtained for the moments of the Riemann zeta function by V. Chandee \cite{Chandee} and for the moments of all Dirichlet $L$-functions modulo $q$ by M. Munsch  \cite{Munsch}.

  We now give two consequences of Theorem \ref{upperbound}. First, by setting $z_1=z_2=0$ in Theorem \ref{upperbound}, we deduce immediately the following upper bounds for moments of the central values of the family $\mathcal F$ given in \eqref{Lfamily}.
\begin{corollary}
\label{cor: upperbound}
Assume GRH for $\zeta_K(s)$ and $L(s, \chi_{(1+i)^5d})$ for all odd, square-free $d$.  For any positive real number $k$ and any $\varepsilon>0$, we have for large $X$,
 \begin{align*}
 \sumstar_{\substack{(d,2)=1 \\ N(d) \leq X}}|L(\tfrac{1}{2},\chi_{(1+i)^5d})|^{k}
  \ll_{k,\varepsilon}X(\operatorname{log}X)^{k(k+1)/2+\varepsilon}.
 \end{align*}
\end{corollary}

  The second consequence concerns upper bounds for shifted fourth moment of the family $\mathcal F$ given in \eqref{Lfamily}, which is what we need in a refined study  in Section \ref{sec: pfmainthm} on the fourth moment under GRH.
\begin{corollary}
\label{4momentupperbound}
Assume GRH for $\zeta_K(s)$ and $L(s, \chi_{(1+i)^5d})$ for all odd, square-free $d$. Let $X$ be large and let $z_1,z_2 \in \mathbb{C}$ with $0 \leq \Re (z_1), \Re (z_2) \leq \frac{1}{\log X}$ and $|\Im (z_1)|, |\Im (z_2)| \leq X$. Then
 \[
  \sumstar_{\substack{(d,2)=1 \\ N(d) \leq X}}  |L(\tfrac{1}{2}+z_1,\chi_{(1+i)^5d})|^2|L(\tfrac{1}{2}+z_2,\chi_{(1+i)^5d})|^2
 \ll X(\operatorname{log}X)^{4+\varepsilon}\left(1+\min \left \{ (\log X)^6, \frac{1}{(|\Im (z_1)|- |\Im (z_2)|)^6} \right \} \right).
 \]
\end{corollary}

  We shall omit the proof of the above corollary as it is analogous to the proof \cite[Theorem 2.4]{Shen}. With the aide of Corollary \ref{4momentupperbound}, we are able to obtain a more precise expression for the fourth moment.  To present our result, we define constants $a_k$ for any real number $k>0$ by
\begin{align}
\label{ak}
a_{k} = 2^{-\frac{k(k+2)}{2}}
 \prod_{\substack{ \varpi \equiv 1 \bmod {(1+i)^3} }} \frac{(1-\frac{1}{N(\varpi)})^{\frac{k(k+1)}{2}}}{1+\frac{1}{N(\varpi)}}
\left(
\frac{
 (1+\frac{1}{\sqrt{N(\varpi)}} )^{-k} +
 (1-\frac{1}{\sqrt{N(\varpi)}} )^{-k}
 }{2}
 +\frac{1}{N(\varpi)}
\right).
\end{align}
 Here and in what follows, we denote $\omega$ for a prime number in $\mathcal{O}_K$, by which we mean that the ideal $(\omega)$ generated by $\omega$ is a prime ideal. We also note that the expression $\varpi \equiv 1 \bmod {(1+i)^3}$ indicates that $\omega$ is a primary element in $\mathcal{O}_K$ defined in Section \ref{sec2.4}.

  Now, we state our result on a conditional asymptotic evaluation on the fourth moment.
\begin{theorem}
\label{main-thm}
Assume GRH for $\zeta_K(s)$ and $L(s, \chi_{(1+i)^5d})$ for all odd, square-free $d$. Then for any $\varepsilon>0$, we have
\[
 \sumstar_{\substack{(d,2)=1 \\ N(d) \leq X}}|L(\tfrac{1}{2},\chi_{(1+i)^5d})|^{4}= \frac{\pi a_{4 }}{2^7 \cdot 3^4 \cdot 5^2 \cdot 7 \cdot \zeta_K(2)} \big ( \frac {\pi}{4} \big )^{10} X{(\log X)}^{10} + O \left(X (\log X)^{9.5+\varepsilon} \right),
\]
Here the ``$*$'' on the sum over $d$ means that the sum is restricted to square-free elements $d$ in $\mathcal{O}_K$ and $a_4$ is defined as in \eqref{ak}.
\end{theorem}

  Without assuming GRH, we  also have the following lower bound for the fourth moment.
\begin{theorem}
\label{theo:mainthm}
 Unconditionally, we have
\begin{align*}
 \sumstar_{(d,2)=1} L(\thalf, \chi_{(1+i)^5d})^4 \Phi\leg{N(d)}{X} \geq  \left( \frac{\pi a_{4 }}{2^7 \cdot 3^4 \cdot 5^2 \cdot 7 \cdot \zeta_K(2)} \big ( \frac {\pi}{4} \big )^{10} +o(1) \right)X(\log X)^{10}.
\end{align*}
 Here the ``$*$'' on the sum over $d$ means that the sum is restricted to square-free elements $d$ in $\mathcal{O}_K$ and $a_4$ is defined as in \eqref{ak}.
\end{theorem}

  Theorems \ref{main-thm} and \ref{theo:mainthm} are similar to those of Shen given in \cite[Theorems 1.1-1.2]{Shen} on the fourth moment of the family of quadratic Dirichlet $L$-functions as well as the result of Soundararajan and Young given in \cite[Theorem 1.1]{S&Y} on the second moment of the family of quadratic twists of modular $L$-functions. The proof for Theorems \ref{main-thm} and \ref{theo:mainthm} given in Section \ref{sec: pfmainthm}, also proceed along the same lines of the proofs of \cite[Theorems 1.1-1.2]{Shen} and \cite[Theorem 1.1]{S&Y}.

\section{Preliminaries}
\label{sec 2}

As a preparation, we first include some auxiliary results needed in the proofs of our theorems.

\subsection{Quadratic residue symbol and quadratic Gauss sum}
\label{sec2.4}
   Recall that $K=\mq(i)$ and it is well-known that $K$ have class number one. We denote $U_K=\{ \pm 1, \pm i \}$ and $D_{K}=-4$ for the group of units in $\mathcal{O}_K$ and the discriminant of $K$, respectively. 

  Every ideal in $\mathcal{O}_K$ co-prime to $2$ has a unique generator congruent to $1$ modulo $(1+i)^3$ which is called primary. It follows from Lemma
6 on \cite[p. 121]{I&R} that an element $n=a+bi \in \mathcal{O}_K$ with $a, b \in \mz$ is primary if and only if $a \equiv 1 \pmod{4}, b \equiv
0 \pmod{4}$ or $a \equiv 3 \pmod{4}, b \equiv 2 \pmod{4}$.

   For $n \in \mathcal{O}_{K}, (n,2)=1$, we denote the symbol  $\leg {\cdot}{n}$ for the quadratic residue symbol modulo $n$ in $K$. For a prime $\varpi \in \mz[i]$
with $N(\varpi) \neq 2$, the quadratic symbol is defined for $a \in
\mathcal{O}_{K}$, $(a, \varpi)=1$ by $\leg{a}{\varpi} \equiv
a^{(N(\varpi)-1)/2} \pmod{\varpi}$, with $\leg{a}{\varpi} \in \{
\pm 1 \}$. When $\varpi | a$, we define
$\leg{a}{\varpi} =0$.  Then the quadratic symbol is extended
to any composite $n$ with $(N(n), 2)=1$ multiplicatively. We further define $\leg {\cdot}{c}=1$ for $c \in U_K$. 

  The following quadratic reciprocity law (see \cite[(2.1)]{G&Zhao4}) holds for two co-prime primary elements $m, n \in \mathcal{O}_{K}$:
\begin{align}
\label{quadrec}
 \leg{m}{n} = \leg{n}{m}.
\end{align}

    Moreover, we deduce from Lemma 8.2.1 and Theorem 8.2.4 in \cite{BEW} that the following supplementary laws hold for primary $n=a+bi$ with $a, b \in \mz$:
\begin{align}
\label{supprule}
  \leg {i}{n}=(-1)^{(1-a)/2} \qquad \mbox{and} \qquad  \hspace{0.1in} \leg {1+i}{n}=(-1)^{(a-b-1-b^2)/4}.
\end{align}

 For any complex number $z$, we define\begin{align*}
 \widetilde{e}(z) =\exp \left( 2\pi i  \left( \frac {z}{2i} - \frac {\bar{z}}{2i} \right) \right) .
\end{align*}
  For any $r\in \mathcal{O}_K$, we define the quadratic Gauss sum $g(r, \chi)$ associated to any quadric Hecke character $\chi$ modulo $q$ of trivial infinite type and the quadratic Gauss sum $g(r, n)$ associated to the quadratic residue symbol $\leg {\cdot}{n}$ for any $(n,2)=1$ by
\begin{align}
\label{g2}
 g(r,\chi) = \sum_{x \bmod{q}} \chi(x) \widetilde{e}\leg{rx}{q}, \quad g(r,n) = \sum_{x \bmod{n}} \leg{x}{n} \widetilde{e}\leg{rx}{n}.
\end{align}
  When $r=1$,  we shall denote $g(\chi)$ for $g(1, \chi)$ and $g(n)$ for $g(1, n)$.  Recall from \cite[(2.2)]{G&Zhao3} that for primary $n$, we have
\begin{align}
\label{gn}
   g(n)=\leg {i}{n}N(n)^{1/2}.
\end{align}

  A Hecke character $\chi$ is said to be primitive modulo $q$ if it does not factor through $\left (\mathcal{O}_K / (q') \right )^{\times}$ for any  divisor $q'$ of $q$ such that $N(q')<N(q)$. Recall from Section \ref{sec 1} that we denote $\chi_c$ for the the quadratic residue symbol $\leg {c}{\cdot}$ and we define $\chi_c(n)$ to be $0$ when $1+i|n$. In Section 2.1 of \cite{G&Zhao4}, it is shown that the symbol $\chi_{i(1+i)^5d}$ defines a primitive quadratic Hecke character modulo $(1+i)^5d$ of trivial infinite type for any odd and square-free $d \in \mathcal{O}_K$. When replacing $d$ by $i^3d$,  we see that the symbol $\chi_{(1+i)^5d}$ also defines a primitive quadratic Hecke character modulo $(1+i)^5d$ of trivial infinite type for any odd and square-free $d \in \mathcal{O}_K$. Our next lemma evaluates $g(\chi_{(1+i)^5d})$ exactly.
\begin{lemma}
\label{lem: primquadGausssum}
  For any odd, square-free $d \in \mathcal{O}_K$, we have
\begin{align}
\label{primquadGausssum}
 g(\chi_{(1+i)^5d})=N((1+i)^5d )^{1/2}.
\end{align}
\end{lemma}
\begin{proof}
   It suffices to prove \eqref{primquadGausssum} with $d$ replaced by $jd$, where $j=1$ or $i$ and $d$ is primary and square-free. It follows from the Chinese remainder theorem that $x = j(1+i)^5y +d z$ varies over the residue class modulo $(1+i)^3d$ when $y$ and $z$ vary over the residue class modulo $d$ and $(1+i)^5$, respectively.  We then deduce that
\begin{align*}
 g(\chi_{j(1+i)^5d})=& \sum_{z \bmod{(1+i)^5}}\sum_{y \bmod{d}} \leg{j(1+i)^5}{j(1+i)^5y +d z}\leg{d}{j(1+i)^5y +d z} \widetilde{e}\leg{jy}{d}\widetilde{e}\leg{z}{(1+i)^5}.
\end{align*}
   As $\chi_{j(1+i)^5}$ is a Hecke character of trivial infinite type  modulo $(1+i)^5$, we deduce that
\begin{align}
\label{1+ireci}
  \leg{j(1+i)^5}{j(1+i)^5y +d z}=\chi_{j(1+i)^5}(j(1+i)^5y +d z)=\chi_{j(1+i)^5}(d z).
\end{align}

   On the other hand, we denote $s(z)$ to be the unique element in $U_K$ such that $s(z)z$ is primary for any $(z,2)=1$. It follows from the quadratic reciprocity law \eqref{quadrec} that
\begin{align}
\label{varpireci}
 \leg{d}{j(1+i)^5y +d z}=\leg {s(z)(1+i)^5jy}{d}.
\end{align}

   We then conclude from \eqref{1+ireci} and \eqref{varpireci} that
\begin{align}
\label{gexpression}
\begin{split}
 g(\chi_{(1+i)^5d})=& \sum_{z \bmod{(1+i)^5}}\sum_{y \bmod{d}} \leg{j(1+i)^5}{d z}\leg {s(z)(1+i)jy}{d} \widetilde{e}\leg{jy}{d}\widetilde{e}\leg{z}{(1+i)^5} \\
=& \sum_{z \bmod{(1+i)^5}}\leg{j(1+i)^5}{z} \leg {s(z)j}{d} \widetilde{e}\leg{z}{(1+i)^5}  \sum_{y \bmod{d}}\leg {jy}{d} \widetilde{e}\leg{jy}{d} \\
 =& N(d)^{1/2}\sum_{z \bmod{(1+i)^5}}\leg{j(1+i)^5}{z} \leg {s(z)ij}{d} \widetilde{e}\leg{z}{(1+i)^5},
\end{split}
\end{align}
  where the last equality above follows from \eqref{gn}.

   In order to evaluate the last sum in \eqref{gexpression}, we note that it suffices to take $z$ to vary over the reduced residue class modulo $(1+i)^5$. One representation of such class consists of the following $16$ elements (note that $\pm 1, \pm i$ consists of the reduced residue class modulo $(1+i)^3$ and $0,1$ consists of the residue class modulo $1+i$):
\begin{align*}
  \{ \pm 1, \pm i \} + l(1+i)^3+k(1+i)^4, \quad l \in \{0,1 \},  k \in \{0,-1 \}.
\end{align*}
  We further write $d=a+bi$ with $a, b \in \mz$ (recall that $a \equiv 1 \pmod{4}, b \equiv 0 \pmod{4}$ or $a \equiv 3 \pmod{4}, b \equiv 2 \pmod{4}$) and check by direct calculations using \eqref{supprule} to see that \eqref{primquadGausssum} is valid with $d$ replaced by $jd$, where $j=1$ or $i$ and $d$ is primary and square-free. This completes the proof of the lemma. 
\end{proof}

  Let $\varphi_{[i]}(n)$ denote the number of elements in the reduced residue class of $\mathcal{O}_K/(n)$,
we recall from \cite[Lemma 2.2]{G&Zhao4} the following explicitly evaluations of $g(r,n)$ for primary $n$.
\begin{lemma} \label{Gausssum}
\begin{enumerate}[(i)]
\item  We have
\begin{align*}
 g(rs,n) & = \overline{\leg{s}{n}} g(r,n), \qquad (s,n)=1, \\
   g(k,mn) & = g(k,m)g(k,n),   \qquad  m,n \text{ primary and } (m , n)=1 .
\end{align*}
\item Let $\varpi$ be a primary prime in $\mathcal{O}_K$. Suppose $\varpi^{h}$ is the largest power of $\varpi$ dividing $k$. (If $k = 0$ then set $h = \infty$.) Then for $l \geq 1$,
\begin{align*}
g(k, \varpi^l)& =\begin{cases}
    0 \qquad & \text{if} \qquad l \leq h \qquad \text{is odd},\\
    \varphi_{[i]}(\varpi^l) \qquad & \text{if} \qquad l \leq h \qquad \text{is even},\\
    -N(\varpi)^{l-1} & \text{if} \qquad l= h+1 \qquad \text{is even},\\
    \leg {ik\varpi^{-h}}{\varpi}N(\varpi)^{l-1/2} \qquad & \text{if} \qquad l= h+1 \qquad \text{is odd},\\
    0, \qquad & \text{if} \qquad l \geq h+2.
\end{cases}
\end{align*}
\end{enumerate}
\end{lemma}

\subsection{The approximate functional equation}
\label{sect: apprfcneqn}

   Let $\chi$ be a primitive quadratic Hecke character modulo $m$ of trivial infinite type of $K$. Let
\begin{align}
\label{Lambda}
  \Lambda(s, \chi) = (|D_K|N(m))^{s/2}(2\pi)^{-s}\Gamma(s)L(s, \chi).
\end{align}
  A well-known result of E. Hecke shows that $L(s, \chi)$ has an
analytic continuation to the whole complex plane and satisfies the
functional equation (see \cite[Theorem 3.8]{iwakow})
\begin{align}
\label{fneqn}
  \Lambda(s, \chi) = W(\chi)(N(m))^{-1/2}\Lambda(1-s, \chi),
\end{align}
   where $|W(\chi)|=(N(m))^{1/2}$.

   When we take $\chi=\chi_{(1+i)^5d}$ for any odd, square-free $d \in \mathcal{O}_K$, it follows from \cite[Theorem 3.8]{iwakow} that we have
$W(\chi_{(1+i)^5d})=g(\chi_{(1+i)^5d})$, so that Lemma \ref{lem: primquadGausssum} implies that in this case the functional equation becomes
\begin{align*}
  \Lambda(s, \chi_{(1+i)^5d}) = \Lambda(1-s, \chi_{(1+i)^5d}).
\end{align*}

   For $n \in \mathcal{O}_{K}$ and rational integer $k \geq 1$, we let $d_{[i],k }(n)$ denote the analogue on $\mathcal{O}_K$ of the usual function $d_k$ on $\mz$. Thus $d_{[i],k }(n)$ equals the coefficient of $N(n)^{-s}$ in the Dirichlet series expansion of the $k$-th power of $\zeta_K(s)$. We shall also write $d_{[i]}(n)$ for $d_{[i],2}(n)$. In particular, when $n$ is primary, we have $d_{[i],1}(n)=1$ and
\begin{align*}
  d_{[i]}(n)=\sum_{\substack{d \equiv 1 \bmod {(1+i)^3} \\ d | n}}1.
\end{align*}

  We denote also for rational integer $j \geq 1$ and any real number $t>0$,
\begin{align}
\label{eq:Vdef}
 V_j(t) = \frac{1}{2 \pi i} \int\limits\limits_{(2)}  w_j(s) t^{-s} \frac {ds}{s}, \quad w_j(s) = \left(\frac{2^{5/2}}{\pi}\right)^{js}
\left ( \frac {\Gamma(\frac{1}{2}+s)}{\Gamma(\frac{1}{2})} \right )^j.
\end{align}
 We shall also write $V(t), w(t)$ for $V_2(t), w_2(t)$ respectively in the rest of the paper.

   By setting $j=1,2, \alpha_1=\alpha_2=0, G_j(s)=1$ in \cite[Lemma 2.6]{Gao1}, we obtain the following approximate functional equation for $L(\half, \chi_{(1+i)^5d})^j$. Note here that the derivation of  \cite[Lemma 2.6]{Gao1} assumes a rapid decay of $G_j(s)$ but the proof also carries over to the case $G_j(s)=1$ due to the rapid decay of $w_j(s)$.
\begin{lemma}[Approximate functional equation]
\label{lem:AFE}
  For any odd, square-free $d \in \mathcal{O}_K$, we have for $j=1,2$,
\begin{align}
\label{fcneqnL}
\begin{split}
 L(\half, \chi_{(1+i)^5d})^j = & 2\sum_{\substack{n \equiv 1 \bmod {(1+i)^3}}} \frac{\chi_{(1+i)^5d}(n)d_{[i],j}(n)}{N(n)^{\frac{1}{2}}} V_j
\left(\frac{ N(n)}{N(d)^{j/2}} \right).
\end{split}
\end{align}
\end{lemma}

\subsection{Poisson summation}
\label{sec Poisson}

   In this section we gather some Poisson summation formulas over $K$.
   We first recall that the Mellin transform $\hat{f}$ for any function $f$ is defined to be
\begin{align*}
     \widehat{f}(s) =\int\limits^{\infty}_0f(t)t^s\frac {\dif t}{t}.
\end{align*}

   We now state a formula for smoothed character sums over all elements in $\mathcal{O}_K$.
\begin{lemma}
\label{Poissonsum} Let $n \in \mathcal{O}_K$ and let $\chi$ be a Hecke character $\pmod {n}$ of trivial infinite type. For any smooth function $W:\mr^{+} \rightarrow \mr$ of compact support,  we have for $X>0$,
\begin{align}
\label{PoissonsumQw}
   \sum_{m \in \mathcal{O}_K}\chi(m)W\left(\frac {N(m)}{X}\right)=\frac {X}{N(n)}\sum_{k \in
   \mathcal{O}_K}g(k,\chi)\widetilde{W} \left(\sqrt{\frac {N(k)X}{N_K(n)}}\right).
\end{align}
    The above is also valid when we replace $\chi$ by $\leg {\cdot}{n}$ and $g(k, \chi)$ by $g(k, n)$. Here $g(k, \chi), g(k, n)$ are defined in \eqref{g2} and
\begin{align}
\label{Wtdef}
   \widetilde{W}(t) =& \int\limits^{\infty}_{-\infty}\int\limits^{\infty}_{-\infty}W(N(x+yi))\widetilde{e}\left(- t(x+yi)\right)\dif x \dif y, \quad t \geq 0.
\end{align}
   Moreover,  the function $\widetilde{W}(t)$ is real-valued for all $t \geq 0$ and when $t>0$, we have for $c_s=\varepsilon>0$,
\begin{align}
\label{WMell}
   \widetilde{W}(t)=&
 \frac {\pi}{2\pi i}
 \int\limits\limits_{(c_s)}\widehat{W}(1-s)
(\pi t)^{-2s}\frac{\Gamma (s)}{\Gamma (1-s)} ds.
\end{align}
\end{lemma}
\begin{proof}
    This lemma is essentially \cite[Lemma 2.7]{G&Zhao4} except for the last assertion. To establish it, we evaluate \eqref{Wtdef} using polar coordinates to see that
\begin{align}
\label{Wt}
     \widetilde{W}(t)  =\int\limits_{\mr^2}\cos (2\pi t y)W(x^2+y^2) \ \dif x \dif y=4\int\limits^{\pi/2}_0\int\limits^{\infty}_0\cos (2\pi t r\sin \theta)W(r^2) \ r \dif r \dif \theta    = 2\int\limits^{\pi/2}_0 \int\limits^{\infty}_0\cos (2\pi t r^{1/2}\sin \theta)W(r) \ \dif r \dif \theta.
\end{align}
  The first equality above shows that $\widetilde{W}(t) \in \mr$ for all $t \geq 0$.

   We now apply inverse Mellin transform to write $W(t)$ as
\begin{align*}
    W \left(t \right)=\frac 1{2\pi i}\int\limits\limits_{(c_u)}\widehat{W}(u)t^{-u}du,
\end{align*}
   where $c_u=\varepsilon>0$. It follows from this and \eqref{Wt} that
\begin{align*}
     \widetilde{W}(t) =& 2\int\limits^{\pi/2}_0\int\limits^{\infty}_0\cos (2\pi t r^{1/2}\sin \theta)\frac 1{2\pi i}\int\limits\limits_{(c_u)}
\widehat{W}(1+u)r^{-u}
\frac {\dif r}{r} \dif \theta.
\end{align*}

   We make some changes of variables (first $r^{1/2} \to r$,
then $2 \pi tr \sin \theta \to r$) to see that
\begin{align}
\label{FnFT}
\begin{split}
  \widetilde{W}(t)  =& \frac 4{2\pi i}
\int\limits\limits_{(c_u)}\widehat{W}(1+u)\int\limits^{\pi/2}_0
\int\limits^{\infty}_0\cos (r)\left(\frac r{2\pi t \sin \theta}\right )^{-2u} \frac {\dif r}{r} \dif \theta du    \\
     =&
 \frac {4}{2\pi i}
 \int\limits\limits_{(c_u)}\widehat{W}(1+u)
(2\pi t)^{2u}\left ( \int\limits^{\pi/2}_0 (\sin \theta )^{2u} \dif \theta
\int\limits^{\infty}_0\cos (r)r^{-2u}\frac {\dif r}{r} \right ) du \\
=&
 \frac {\pi}{2\pi i}
 \int\limits\limits_{(c_u)}\widehat{W}(1+u)
(\pi t)^{2u}\frac{\Gamma (-u)}{\Gamma (1+u)} du,
\end{split}
\end{align}
  where the last line above follows from the relation (see \cite[Section 2.4]{Gao1}) that
\begin{align*}
\int\limits^{\pi/2}_0 (\sin \theta )^{-u} \dif \theta
\int\limits^{\infty}_0\cos (r)r^{u}\frac {\dif r}{r}=\frac {\pi}{2}2^{u-1}\frac{\Gamma \leg{u}{2}}{\Gamma\leg{2-u}{2}}.
\end{align*}

  By a further change of variable $u \rightarrow -s$ in the last integral of \eqref{FnFT}, we can recast $\widetilde{W}(t)$ as
\begin{align}
\label{Phitilde}
\begin{split}
  \widetilde{W}(t)   =&
 \frac {\pi}{2\pi i}
 \int\limits\limits_{(c_s)}\widehat{W}(1-s)
(\pi t)^{-2s}\frac{\Gamma (s)}{\Gamma (1-s)} ds,
\end{split}
\end{align}
  where we can retake $c_s=\varepsilon$ as well and this completes the proof.
\end{proof}

   We remark that when $\chi$ is a primitive Hecke character, we have
\begin{align*}
   g(r, \chi) = \overline{\chi} (r) g(\chi).
\end{align*}
  It follows from this and the expression given in \eqref{WMell} that the formula given \eqref{PoissonsumQw} is equivalent to a version of the Poisson summation
formula over number fields by L. Goldmakher and B. Louvel in \cite[Lemma 3.2]{G&L} for the case of the Gaussian field.

   In the proof of Theorems \ref{main-thm} and \ref{theo:mainthm}, we need to consider a smoothed character sum over odd algebraic integers in $\mathcal{O}_K$. For this, we quote the following Poisson summation formula from \cite[Corollary 2.8]{G&Zhao4}, which is a consequence of Lemma \ref{Poissonsum} above.
\begin{lemma}
\label{Poissonsumformodd} Let $n \in \mathcal{O}_K$ be primary and $\leg {\cdot}{n}$ be the quadratic residue symbol $\pmod {n}$. For any smooth function $W:\mr^{+} \rightarrow \mr$ of compact support,  we have for $X>0$,
\begin{align*}
   \sum_{\substack {m \in \mathcal{O}_K \\ (m,1+i)=1}}\leg {m}{n} W\left(\frac {N(m)}{X}\right)=\frac {X}{2N(n)}\leg {1+i}{n}\sum_{k \in
   \mathcal{O}_K}(-1)^{N(k)} g(k,n)\widetilde{W}\left(\sqrt{\frac {N(k)X}{2N(n)}}\right).
\end{align*}
\end{lemma}

\subsection{Analytical behaviors of certain Dirichlet series}
\label{sect: alybehv}

   In this section, we discuss the analytical behaviors of several Dirichlet series that are needed in the proof of Theorems \ref{main-thm} and \ref{theo:mainthm}.
We first define for $\Re (\alpha),\Re (\beta) > \frac{1}{2}$,
\begin{align}
\label{zalphabeta}
  Z(\alpha, \beta) =\sum_{\substack{n_1, n_2 \equiv 1 \bmod {(1+i)^3}\\  n_1n_2=\square}} \frac{d_{[i]}(n_1)d_{[i]}(n_2)}{N(n_1)^{\alpha}N(n_2)^{\beta}}\mathcal{P}(n_1n_2),
\end{align}
  where we define $\mathcal{P}(n)$ for primary $n$ by
\begin{align}
\label{Pn}
 \mathcal{P}(n)=\prod_{\substack{ \varpi \equiv 1 \bmod {(1+i)^3} \\ \varpi |n}} \leg {N(\varpi)}{N(\varpi)+1}.
\end{align}

Our first result concerns the analytical behaviors of $Z(\alpha, \beta)$. The proof is similar to \cite[Lemma 4.1]{Shen}, so we omit it here.
\begin{lemma}
\label{lem: zalphabeta}
For $\Re (\alpha),\Re (\beta) > \frac{1}{2}$, we have
\begin{align*}
 Z(\alpha, \beta)
= \zeta_K^3 (2\alpha) \zeta_K^3 (2\beta) \zeta_K^4 (\alpha+\beta) Z_1(\alpha,\beta),
\end{align*}
where 
\begin{align*}
 Z_1(\alpha,\beta) = \left( 1- \frac{1}{4^{\alpha}}\right)^3 \left( 1- \frac{1}{4^{\beta}}\right)^3 \left( 1- \frac{1}{2^{\alpha+\beta}}\right)^4\prod_{\substack{ \varpi \equiv 1 \bmod {(1+i)^3} \\ \varpi |n_1n_2}} Z_{1,\varpi}(\alpha,\beta),
\end{align*}
 and where for primary $\varpi$,
\begin{align*}
& Z_{1,\varpi}(\alpha,\beta) \\
=&  \left( 1- \frac{1}{N(\varpi)^{2\alpha}}\right)  \left( 1- \frac{1}{N(\varpi)^{2\beta}}\right)  \left( 1- \frac{1}{N(\varpi)^{\alpha+\beta}}\right)^4\Bigg (
1 + \frac{4}{N(\varpi)^{\alpha+\beta}} + \frac{1}{N(\varpi)^{2\alpha}} + \frac{1}{N(\varpi)^{2\beta}} + \frac{1}{N(\varpi)^{2\alpha+2\beta}} \\
&- \frac{1}{N(\varpi)+1}\Big (
\frac{3}{N(\varpi)^{2\alpha}}
+ \frac{3}{N(\varpi)^{2\beta}} + \frac{4}{N(\varpi)^{\alpha+\beta}}  -  \frac{1}{N(\varpi)^{4\alpha}} -  \frac{1}{N(\varpi)^{4\beta}} \\
& - \frac{3}{N(\varpi)^{2\alpha+2\beta}} + \frac{2}{N(\varpi)^{2\alpha+4\beta}} + \frac{2}{N(\varpi)^{4\alpha+2\beta}}  - \frac{1}{N(\varpi)^{4\alpha+4\beta}}
\Big )
\Bigg ) .
\end{align*}
 Moreover, $Z_1(\alpha,\beta)$ is analytic and uniformly bounded in the region $\Re (\alpha),\Re (\beta) \geq \frac{1}{4}+\varepsilon$.
\label{lem:Z_1}
\end{lemma}

Let $g(k,n)$ be defined as in \eqref{g2}. We now fix a generator for every prime ideal $(\varpi) \in \mathcal{O}_K$ by taking $\varpi$ to be primary if $(\varpi, 1+i)=1$ and $1+i$ for the ideal $(1+i)$ (noting that $(1+i)$ is the only prime ideal in $\mathcal{O}_K$ that lies above the integral ideal $(2) \in \mz$). We also fix  $1$ as the generator for the ring $\mz[i]$ itself and
extend the choice of the generator for any ideal of $\mathcal{O}_K$ multiplicatively. We denote the set of such generators by $G$. For any $k \in \mathcal{O}_K$, we shall hence denote $k_1, k_2$ to be the unique pair of elements in $\mathcal{O}_K$ such that $k=k_1k^2_2$ with $k_1$ being square-free and $k_2 \in G$. In this way, we define
\begin{align*}
  Z(\alpha,\beta,a,k) & =\sum_{\substack{n_1, n_2 \equiv 1 \bmod {(1+i)^3} \\ (n_1n_2,a)=1}}
\frac {d_{[i]}(n_1)d_{[i]}(n_2)}{N(n_1)^{\alpha}N(n_2)^{\beta}} \frac {g(k_1k^2_2,n_1n_2)}{N(n_1n_2)}.
\end{align*}

  Our next lemma gives the analytic properties of $Z(\alpha,\beta,a,k) $, we omit the proof here since it is similar to \cite[Lemma 5.2]{Shen}. 
\begin{lemma}
\label{lem:15.2}
For any $k \in \mathcal{O}_K$, let  $k=k_1k_2^2$ with $k_1$ square-free and $k_2 \in G$. Then for $\Re(\alpha),\Re(\beta)>\frac{1}{2}$, we have
\begin{align*}
Z(\alpha,\beta,a,k)  = L^2(\tfrac{1}{2}+\alpha,\chi_{ik_1}) L^2(\tfrac{1}{2}+\beta,\chi_{ik_1}) Z_2(\alpha,\beta,a,k),
\end{align*}
 where
\begin{align*}
Z_2(\alpha,\beta,a,k)=Z_{2,1+i}(\alpha,\beta,a,k)\prod_{\substack{\varpi \equiv 1 \bmod {(1+i)^3}}}Z_{2,\varpi}(\alpha,\beta,a,k).
\end{align*}
 Here for $\varpi|2a$,
\begin{align*}
Z_{2,\varpi}(\alpha,\beta,a,k) = \left( 1-\frac{\chi_{ik_1}(\varpi)}{N(\varpi)^{\frac{1}{2}+\alpha}}\right)^2\left( 1-\frac{\chi_{ik_1}(\varpi)}{N(\varpi)^{\frac{1}{2}+\beta}}\right)^2,  
\end{align*}
and for $\varpi \nmid 2a$, 
\begin{align*}
Z_{2,\varpi}(\alpha,\beta,a,k) = \left( 1-\frac{\chi_{ik_1}(\varpi)}{N(\varpi)^{\frac{1}{2}+\alpha}}\right)^2\left( 1-\frac{\chi_{ik_1}(\varpi)}{N(\varpi)^{\frac{1}{2}+\beta}}\right)^2 \sum_{n_1=0}^{\infty} \sum_{n_2=0}^{\infty} \frac{d_{[i]}(\varpi^{n_1})d_{[i]}(\varpi^{n_2})}{N(\varpi)^{n_1\alpha + n_2\beta}}\frac{g(k,\varpi^{n_1+n_2})}{N(\varpi)^{n_1+n_2}}. 
\end{align*}
 Moreover, $Z_2(\alpha,\beta,a,k)$ is analytic in the region $\Re(\alpha),\Re(\beta)>0$ and for $\Re(\alpha),\Re(\beta) \geq \frac{1}{\log X}$, we have
\begin{align}
\label{equ:15.111111}
Z_2(\alpha,\beta,a,k) \ll d^4_{[i]}(a)  d^{12}_{[i]} (k) (\log X)^{10},
\end{align} 
where the implied constant is absolute.
\end{lemma}

 We define for a prime $\omega \in \mathcal{O}_K$, $\alpha, \beta, \gamma \in \mc$, 
\begin{align*}
K_1(\alpha,\beta,\gamma;\varpi) =& \left( 1- \frac{1}{N(\varpi)^{\frac{1}{2}+\alpha}}\right)^2  \left( 1- \frac{1}{N(\varpi)^{\frac{1}{2}+\beta}}\right)^2  \left( 1- \frac{1}{N(\varpi)^{2\alpha + 2\gamma}}\right)^2  \left( 1- \frac{1}{N(\varpi)^{2\beta + 2\gamma}}\right)^2, \\
\begin{split}
 K_2(\alpha,\beta,\gamma;\varpi)
=& \left( 1- \frac{1}{N(\varpi)^{\frac{1}{2}+\alpha}}\right)^2  \left( 1- \frac{1}{N(\varpi)^{\frac{1}{2}+\beta}}\right)^2  \Bigg( \left(1-\frac{1}{N(\varpi)}\right)\left(1+\frac{1}{N(\varpi)^{2\alpha+2\gamma}}\right)\left(1+\frac{1}{N(\varpi)^{2\beta+2\gamma}}\right) \\
&+ \frac{1}{N(\varpi)}\left(1-\frac{1}{N(\varpi)^{2\alpha+2\gamma}}\right)^2\left(1-\frac{1}{N(\varpi)^{2\beta+2\gamma}}\right)^2
+ \left(1-\frac{1}{N(\varpi)}\right)\frac{4}{N(\varpi)^{\alpha+\beta+2\gamma}}\\
&+
2\left( 1-\frac{1}{N(\varpi)^{2\gamma}}\right)\left( \frac{1}{N(\varpi)^{\frac{1}{2}+\alpha}} + \frac{1}{N(\varpi)^{\frac{1}{2}+\beta}} + \frac{1}{N(\varpi)^{\frac{1}{2}+2\alpha+\beta + 2\gamma}}+ \frac{1}{N(\varpi)^{\frac{1}{2}+\alpha+2\beta + 2\gamma}}\right)\Bigg ).
\end{split}
\end{align*}

  With the above notations, we note the following result concerning a Dirichlet series related to $Z_2(\alpha,\beta,a,k)$. 
\begin{lemma}
\label{lem:15.3}
For $\Re(\alpha),\Re(\beta)>0, \Re(\gamma)>\tfrac{1}{2}$, we have
\begin{align}
\label{equ:Z_3}
\sum_{k \in G} \frac{(-1)^{N(k)}}{N(k)^{2\gamma}}Z_2(\alpha,\beta,a,\pm i k^2) =
(2^{1-2 \gamma}-1)\zeta_K(2 \gamma) \zeta_K^2(2 \alpha +2 \gamma) \zeta_K^2(2 \beta +2 \gamma) Z_3(\alpha,\beta,\gamma,a),
\end{align}
where
\begin{align*}
Z_3(\alpha,\beta,\gamma,a) = \prod_{\substack{\varpi \equiv 1 \bmod {(1+i)^3} }} Z_{3,\varpi}(\alpha,\beta,\gamma,a),
\end{align*}
 Here for $\varpi |2a$, $Z_{3,\varpi}(\alpha,\beta,\gamma,a) = K_1(\alpha,\beta,\gamma;\varpi)$
and for $\varpi \nmid 2a$, $Z_{3,\varpi}(\alpha,\beta,\gamma,a)= K_2(\alpha,\beta,\gamma;\varpi)$. 

In addition, we have
\begin{itemize}
\item[\textit{(1)}] $Z_3(\alpha,\beta,\gamma,a)$  is analytic and uniformly bounded for $\Re(\alpha),\Re(\beta)\geq \tfrac{1}{2}+\varepsilon, \Re(\gamma)\geq 2\varepsilon$.

\item[\textit{(2)}] $Z_3(\alpha,\beta,\gamma,a) \ll  (\log X)^{10} $ for $\Re(\alpha),\Re(\beta)\geq \tfrac{1}{2}+\frac{1}{\log X}, \Re(\gamma)\geq \frac{2}{\log X}$ with the implied constant being absolute.
\end{itemize}
\end{lemma}
\begin{proof}
   We let $f(k)=g(k,n)/N(k)^s$ and we divide the sum over $k$ in \eqref{equ:Z_3} into two sums, according to $(k, 1+i)=1$ or not, to see that
\begin{align*}
   \sum_{k \in G}(-1)^{N(k)}f(\pm i k^2)  =\left (2\sum_{k \in G}f(\pm 2i k^2)-\sum_{k \in G}f(\pm i k^2) \right ),
\end{align*}
   Note that when $(n, 1+i)=1$, $g(k,n)=g(2k,n)$ by Lemma \ref{Gausssum}. It follows that we have $f(\pm 2 i k^2)=4^{-s}f(\pm i k^2)$ so that
\begin{align*}
   \sum_{k \in G}(-1)^{N(k)}f(\pm i k^2)  = (2^{1-2s}-1)\sum_{k \in G}f(\pm ik^2).
\end{align*}

    We then deduce from this and the definition of $Z_2$ given in Lemma \ref{lem:15.2} that
\begin{align*}
\sum_{k \in G} \frac{(-1)^{N(k)}}{N(k)^{2\gamma}}Z_2(\alpha,\beta,a, \pm i k^2)  =&  (2^{1-2 \gamma}-1)\sum_{k \in G} \frac{1}{N(k)^{2\gamma}}Z_2(\alpha,\beta,a,k^2) =  \prod_{\substack{\varpi \in G}} \sum_{b=0}^{\infty} \frac{Z_{2,\varpi}(\alpha,\beta,a,\varpi^{2b})}{N(\varpi)^{2b\gamma}},
\end{align*}
  where the last equality above follows from the observation that $g(\pm i k^2, n)$ is multiplicative with respect to $k$ and that we have $g(\pm i k^2, n)=g(k^2, n)$ when $n$ is primary from Lemma \ref{Gausssum}. The assertions of Lemma \ref{lem:15.3} now follows by arguing similarly to the proof of \cite[Lemma 5.3]{Shen}.
\end{proof}

  Lastly, we note the following result which can be established similar to the proof of \cite[Lemma 5.4]{Shen}. 
\begin{lemma}
\label{lem:15.4}
For $\Re(\alpha),\Re(\beta)>\tfrac{1}{2}, 0 < \Re(\gamma) < \tfrac{1}{2}$, we have
\begin{align*}
\sum_{\substack{a \equiv 1 \bmod {(1+i)^3}  }} \frac{\mu_{[i]}(a)}{N(a)^{2-2\gamma}}  Z_3(\alpha,\beta, \gamma, a) = \frac{ \zeta_K(2\alpha+2\gamma)\zeta_K(2\beta+2\gamma) \zeta_K^4(\alpha+\beta+2\gamma)}{\zeta_K^2(\frac{1}{2}+\alpha+2\gamma)\zeta_K^2(\frac{1}{2}+\beta+2\gamma)} Z_4(\alpha,\beta,\gamma),
\end{align*}
where
\begin{align*}
Z_4(\alpha,\beta,\gamma)
=& K_1(\alpha,\beta,\gamma; 1+i)  \prod_{\substack{\varpi \equiv 1 \bmod {(1+i)^3} }} \left( K_2(\alpha,\beta,\gamma;\varpi)-\frac{1}{N(\varpi)^{2-2\gamma}}K_1(\alpha,\beta,\gamma;\varpi) \right)\\
 & \times \prod_{\substack{\varpi \equiv 1 \bmod {(1+i)^3} }} \frac{(1-\frac{1}{N(\varpi)^{2\alpha+2\gamma}})(1-\frac{1}{N(\varpi)^{2\beta+2\gamma}})(1-\frac{1}{N(\varpi)^{\alpha+\beta+2\gamma}})^4}
{ (1-\frac{1}{N(\varpi)^{\frac{1}{2}+\alpha+2\gamma}})^{2} (1-\frac{1}{N(\varpi)^{\frac{1}{2}+\beta+2\gamma}})^{2} }.
\end{align*}

Moreover, $Z_4(\alpha,\beta,\gamma)$ is analytic and uniformly bounded for $\Re(\alpha),\Re(\beta)\geq \frac{3}{8}, -\frac{1}{16}  \leq \Re(\gamma) \leq \frac{1}{8}$.
\end{lemma}

\subsection{A mean value estimate for quadratic Hecke $L$-functions}
 In \cite[Theorem 1]{DRHB}, D. R. Heath-Brown established a powerful quadratic large sieve result for Dirichlet characters.  Such result was extended by K. Onodera in \cite{Onodera} to quadratic residue symbols in the Gaussian field. Applying Onodera's result in a similar fashion as in the derivation of \cite[Theorem 2]{DRHB} by Heath-Brown to obtain a mean value estimation for the fourth moment of the family of primitive quadratic Dirichlet $L$-functions,  we have the following upper bound for the fourth moment of quadratic Hecke $L$-functions unconditionally.
\begin{lemma}
\label{lem:2.3}
Suppose $\sigma+it$ is a complex number with $\sigma \geq \frac{1}{2}$. Then
\begin{align*}
\sumstar_{\substack{(d,2)=1 \\ N(d) \leq X}} |L(\sigma+it,\chi_{(1+i)^5d})| ^4
\ll X^{1+\varepsilon} (1+|t|^2)^{1+\varepsilon}.
\end{align*}
\end{lemma}

\section{Proof of Theorem \ref{theo:lowerbound}}
\label{sec: pfthm1}

   We recall a result of Landau \cite{Landau} implies that for an algebraic number field $F$ of degree $n=2$ and any primitive ideal character $\chi$ of $F$ with
conductor $q$, we have for $X > 1$,
\begin{align}
\label{PVnf}
  \sum_{N_F(I) \leq X} \chi(I) \ll |N_F(q)\cdot D_F|^{1/3}\log^2( |N_F(q)\cdot D_F|)X^{1/3},
\end{align}
where $N_F(q), N_F(I)$ denotes the norm of $q$ and $I$ respectively, $D_F$ denotes the discriminant of $F$ and $I$ runs over integral ideas of $F$.

   We deduce from \eqref{PVnf} by partial summation that for odd, square-free $d$, the series
\begin{align*}
  \sum_{\substack{ n \equiv 1 \bmod {(1+i)^3} }} \frac{\chi_{(1+i)^5d}(n)}{\sqrt{N(n)}}
\end{align*}
   is convergent and equals to $L(\tfrac 12, \chi_{(1+i)^5d})$.  In particular, this implies that $L(\tfrac 12, \chi_{(1+i)^5d}) \in \mr$. It follows from this and the observation that $k$ is an even natural number that we may further restrict the sum over $d$ in \eqref{lowerbound} to satisfy $X/2 < N(d) \le X$.

  We set $x= X^{\frac{1}{10k}}$ and apply H{\" o}lder's inequality to see that
\begin{align}
\label{Llowerbound}
 \sumstar_{\substack{(d,2)=1 \\ X/2 < N(d) \le X }} L(\thalf, \chi_{(1+i)^5d})^k
\ge \frac{\mathcal{S}_1^{k}}{\mathcal{S}_2^{k-1}},
\end{align}
where
\begin{align*}
\mathcal{S}_1 = \sumstar_{\substack{(d,2)=1 \\ X/2 < N(d) \le X }}  L(\tfrac 12, \chi_{(1+i)^5d})A({d})^{k-1},
\ \  \ \
\mathcal{S}_2=\sumstar_{\substack{(d,2)=1 \\ X/2 < N(d) \le X }} A({d})^{k}.
\end{align*}
 and
$$
A(d)= \sum_{\substack{ n \equiv 1 \bmod {(1+i)^3} \\ N(n) \le x }} \frac{\chi_{(1+i)^5d}(n)}{\sqrt{N(n)}}.
$$

In the remaining of the proof, it thus suffices to bound $\mathcal{S}_1$ and $\mathcal{S}_2$. We bound $\mathcal{S}_2$ first by noting that
\begin{align*}
\mathcal{S}_2 \ll \sumstar_{\substack{(d,2)=1}} A({d})^{k}W(\frac {N(d)}{X}),
\end{align*}
  where $W(t)$ is any non-negative smooth function that is supported on $(\frac 1{2}-\varepsilon, 1+\varepsilon)$ for some fixed small $0<\varepsilon<1/2$ such that $W(t) \gg 1$ for $t \in (\half, 1)$.

 We now expand  $A(d)^k$ to see that
$$
\mathcal{S}_2 = \sum_{\substack{ n_1, \ldots, n_k \equiv 1 \bmod {(1+i)^3} \\ N(n_1),\ldots, N(n_k) \le x}} \frac{1}{\sqrt{N(n_1\cdots n_k)}}
\sumstar_{\substack{(d,2)=1}} \leg {(1+i)^5d}{n_1 \cdots n_k}W(\frac {N(d)}{X}).
$$

   We consider the inner sum above by setting $n=n_1\cdots n_k$.  Using M\"obius function to express the condition that $d$ is square-free, we see that
\begin{align}
\label{smoothcharsum}
\sumstar_{\substack{(d,2)=1}} \leg {(1+i)^5d}{n}W(\frac {N(d)}{X})
= \sum_{\substack{ \alpha \equiv 1 \bmod (1+i)^3 \\ N(\alpha) \leq 2X^{1/2}}}   \leg{(1+i)^5\alpha^2}{n}
\sum_{\substack{ (d,2)=1 }} \leg {d}{n}W(\frac {N(\alpha^2 d)}{X}).
\end{align}

 Note that for smoothed sums involving any non-principal Hecke character modulo $a$ of trivial infinite type, we have (see \cite[(1.4)]{G&Zhao2019}) that for $y \geq 1$
\begin{align}
\label{PV}
    \sum_{c \equiv 1 \bmod{ (1+i)^3}} \chi (c)\Phi \left( \frac{N(c)}{y} \right) \ll_{\varepsilon} N(a)^{(1+\varepsilon)/2}.
\end{align}

 Now, we write $d=jd'$ with $j \in U_K$ and $d'$ being primary and apply the quadratic reciprocity law \eqref{quadrec} to see that
\begin{align}
\label{dchar}
  \leg {d}{n}=\leg {j}{n}\leg {n}{d'}.
\end{align}
  Note that if $n$ is not a square then the symbol $\leg {n}{\cdot}$ can be regarded as a  non-principal Hecke character $\pmod {16n}$ of trivial infinite type. By decomposing the sum over $d$ in \eqref{smoothcharsum} into sums over $j, d'$ and apply \eqref{PV} to the sum over $d'$, we deduce that the sum over $d$ in \eqref{smoothcharsum} is $\ll N(n)^{1/2+\varepsilon}$.  On the other hand, the
sum over $d$ is trivially $\ll X/N(\alpha)^2$.  We then conclude that if $n$ is not a square,
\begin{align}
\label{nnotsquare}
\sum_{\substack{(d,2)=1}} \leg {(1+i)^5d}{n}W(\frac {N(\alpha^2 d)}{X}) \ll \sum_{\substack{ \alpha \equiv 1 \bmod (1+i)^3 \\ N(\alpha) \leq 2X^{1/2}}}   \leg{(1+i)^5\alpha^2}{n}
\min \Big( N(n)^{1/2+\varepsilon}, \frac{X}{N(\alpha)^2}\Big)
\ll X^{\frac 12} N(n)^{\frac 14+\varepsilon}.
\end{align}

  If $n$ is a perfect square, then by applying the following result for the Gauss circle problem (see \cite{Huxley1}) with $\theta = 131/416$, 
\begin{align}
\label{Gausscircle}
  \sum_{N(a) \leq x} 1 = \pi x+O(x^{\theta})
\end{align}
 together with a routine argument, we see that
\begin{align}
\label{nsquare}
\sumstar_{\substack{(d,2)=1}} \leg {(1+i)^5d}{n}W(\frac {N(d)}{X}) \ll
\sumstar_{\substack{ N(d) \le X \\ (d,2n)=1}}  1 = \frac{2\pi X}{3\zeta_K(2)}\mathcal{P}(n) + O(X^{\frac 12+\varepsilon}N(n)^{\varepsilon}),
\end{align}
  where $\mathcal{P}(n)$ is defined in \eqref{Pn}.

    Using \eqref{nnotsquare} and \eqref{nsquare}, we see that
\begin{align}
\label{boundS2}
\mathcal{S}_2 \ll X \sum_{\substack{ n_1, \ldots, n_k \equiv 1 \bmod {(1+i)^3} \\ N(n_1),\ldots, N(n_k) \le x \\ n_1 \cdots n_k =n^2 }} \frac{\mathcal{P}(n)}{N(n)}
+ O\Big( X^{\frac 12+\epsilon} x^{k(\frac 34+\epsilon)}\Big).
\end{align}

   We note that
\begin{align}
\label{estmainS2}
& \sum_{\substack{ n \equiv 1 \bmod {(1+i)^3} \\ N(n)^2 \le x }}
\frac{d_{[i],k}(n^2)}{N(n)} \mathcal{P}(n) \ll \sum_{\substack{ n_1, \ldots, n_k \equiv 1 \bmod {(1+i)^3} \\ N(n_1),\ldots, N(n_k) \le x \\ n_1 \cdots n_k = n^2 }} \frac {\mathcal{P}(n)}{N(n)}
\ll \sum_{\substack{ n \equiv 1 \bmod {(1+i)^3} \\ N(n)^2 \le x^k}}
\frac{d_{[i],k}(n^2)}{N(n)}  \mathcal{P}(n).
\end{align}

 Similar to \cite[Theorem 2]{Selberg}, we have that for a positive constant $C(k)$,
\begin{align}
\label{daverage}
\sum_{\substack{ n \equiv 1 \bmod {(1+i)^3} \\ N(n) \le z }}
\frac{d_{[i],k}(n^2)}{N(n)} \mathcal{P}(n)
\sim C(k) (\log z)^{k(k+1)/2}.
\end{align}

 Applying this in \eqref{estmainS2} and setting $x=X^{\frac 1{10k}}$ in \eqref{boundS2}, we deduce that
\begin{align}
\label{S2upperbound}
\mathcal{S}_2 \ll X (\log X)^{k(k+1)/2}.
\end{align}

 We evaluate $\mathcal{S}_1$ next. By applying the approximate
functional equation \eqref{fcneqnL} for $L(\tfrac 12,\chi_{(1+i)^5d})$, we see that
\begin{align}
\label{S1}
\begin{split}
\mathcal{S}_1= 2\sum_{\substack{ n \equiv 1 \bmod {(1+i)^3}}} \frac{1}{\sqrt{N(n)}}  \sum_{\substack{ n_1, \ldots, n_k \equiv 1 \bmod {(1+i)^3} \\ N(n_1),\ldots, N(n_{k-1}) \le x}} \frac{1}{\sqrt{N(n_1\cdots n_{k-1})}}
\sumstar_{\substack{(d,2)=1 \\ X/2 < N(d) \le X }}  \leg {(1+i)^5d}{n n_1 \cdots n_{k-1}}
 V_1\Big(\frac{N(n)}{\sqrt{N(d)}}\Big).
\end{split}
\end{align}

  Here we note that (see \cite[Lemma 2.1]{sound1}) $V_{1}(x)$ is real-valued and smooth on $[0, \infty)$ and the $j$-th derivative of $V_{1}(x)$ satisfies
\begin{align}
\label{2.07}
      V_{1}\left (x \right) = 1+O(\xi^{1/2-\epsilon}) \; \mbox{for} \; 0<\xi<1   \quad \mbox{and} \quad V^{(j)}_{1}\left (\xi \right) =O(e^{-\xi}) \; \mbox{for} \; \xi >0,\; j \geq 0.
\end{align}

  If $n n_1 \cdots n_{k-1}$ is not a square, then using \eqref{PV}, \eqref{2.07} and
partial summation we see that
\begin{align}
\label{sumnonsquare}
 \sumstar_{\substack{(d,2)=1 \\ X/2 < N(d) \le X }}  \leg {(1+i)^5d}{n_1 \cdots n_{k-1}}
 V_1\Big(\frac{N(n)}{\sqrt{N(d)}}\Big)
\ll X^{\frac 12} N(nn_1\cdots n_{k-1})^{\frac 14 +\epsilon}  e^{-N(n)/\sqrt{X}}.
\end{align}
If $n n_1\cdots n_{k-1}$ is an square, then using the right-hand side expression in \eqref{nsquare} together with \eqref{2.07} and partial
summation implies that the sum over $d$ in \eqref{S1} is
\begin{align}
\label{sumsquare}
\begin{split}
& \sumstar_{\substack{(d,2)=1 \\ X/2 < N(d) \le X }}  \leg {(1+i)^5d}{n_1 \cdots n_{k-1}}
 V_1\Big(\frac{N(n)}{\sqrt{N(d)}}\Big) \\
=& \frac{X}{3 \zeta_K(2)} \mathcal{P}(nn_1\cdots n_{k-1})
\int_1^2  V_1\Big(\frac{\sqrt{2} N(n)}{\sqrt{X t}}\Big)  dt + O(X^{\frac 12+\epsilon}N(nn_1\cdots n_{k-1})^{\varepsilon} e^{-N(n)/\sqrt{X}}).
\end{split}
\end{align}
  Applying \eqref{sumnonsquare} and \eqref{sumsquare} in \eqref{S1}, we see that the error terms in \eqref{sumnonsquare} and \eqref{sumsquare}
contribute to \eqref{S1} an amount $\ll X^{\frac 12+\epsilon} x^{(k-1)(\frac 34+\epsilon)} X^{\frac 38+\epsilon}
\ll X^{\frac {39}{40}}$.

  To estimate contribution of the main term in \eqref{sumsquare} to \eqref{S1}, we write $n_1\cdots n_{k-1}$ as $r s^2$ where $r$ and $s$ are
primary and $r$ is square-free.  Then $n$ must be of the form $r\ell^2$ where $\ell$ is primary.
It follows that the contribution of the main term in \eqref{sumsquare} to \eqref{S1} is
$$
\frac{2X}{3\zeta_K(2)}  \sum_{\substack{ r, s \equiv 1 \bmod {(1+i)^3} \\ rs^2 =n_1 \cdots n_{k-1} \\ N(n_1), \ldots, N(n_{k-1}) \le x }}
\frac{1}{N(rs)} \sum_{\substack{ \ell \equiv 1 \bmod {(1+i)^3}}}
\frac{1}{N(\ell)} \int_1^2  \mathcal{P}(rs \ell)V_1\Big( \frac{\sqrt{2} N(r\ell^2) }{\sqrt{Xt}}\Big) dt.
$$
Note that $N(r) \le x^{k-1} < X^{\frac 1{10}}$, and by a standard argument, we see that the sum over $\ell$ above is
$$
= \frac 23 \mathcal{P}(rs)\prod_{\substack{ \varpi \equiv 1 \bmod {(1+i)^3} \\ \varpi \nmid rs}}
\Big(1-\frac{1}{N(\varpi)(N(\varpi+1)} \Big) \frac{1}{4} \cdot \frac{\pi}{4}
\log \frac{\sqrt{X}}{N(r)} + O(1).
$$

It follows that the contribution of the main term in \eqref{sumsquare}  to \eqref{S1} is
\begin{align*}
&\gg X\log X \sum_{\substack{ r, s \equiv 1 \bmod {(1+i)^3} \\ rs^2 =n_1 \cdots n_{k-1} \\ N(n_1), \ldots, N(n_{k-1}) \le x }}
\frac{1}{N(rs)} \mathcal{P}(rs)
\\
&\gg X\log X \sum_{\substack{ r \text{ primary and square-free}\\ s\text{ primary} \\ N(rs^2) \le x }} \frac{d_{[i],k-1}(rs^2)}{N(rs)}
\mathcal{P}(rs) \gg X (\log X) (\log x)^{k-1+k(k-1)/2}.
\end{align*}
 We then deduce that
$$
\mathcal{S}_1 \gg X(\log X)^{k(k+1)/2}.
$$
The assertion of Theorem \ref{theo:lowerbound} now follows from \eqref{Llowerbound}, \eqref{S2upperbound} and the above bound.

\section{Proof of Theorem \ref{upperbound}}
\label{sec:upper bd}

\subsection{A few lemmas}

  We first include a few lemmas needed in the proof of Theorem \ref{upperbound}. We denote $\Lambda_{[i]}(n)$ for the von Mangoldt function on $K$. Thus $\Lambda_{[i]}(n)$ equals the coefficient of $N(n)^{-s}$ in the Dirichlet series expansion of $\zeta^{'}_K(s)/\zeta_K(s)$. Our first lemma provides an upper bound of $\log |L(s, \chi)|$ in terms of a sum involving prime powers.
\begin{lemma}
\label{lem: logLbound}
Let $\chi$ be a non-principal primitive quadratic Hecke character modulo $m$ of trivial infinite type. Assume GRH for $\zeta_K(s)$ and $L(s, \chi)$. Let $T$ be a large number and $x \geq 2$.  Let $\lambda_0=0.56\ldots$
denote the unique positive real number satisfying $e^{-\lambda_0} = \lambda_0$.
For all $\lambda_0 \leq \lambda \leq 5$ we have uniformly for $|t| \leq T$ and $\half \leq \sigma \leq \half+\frac {\lam}{\log x}$ that
\begin{align}
\label{logLupperbound}
\log |L(\sigma+it, \chi)| \le \Re{\sum_{\substack{(n) \\ N(n) \le x }} \frac{\Lam(n)}{N(n)^{\frac 12+ \frac{\lam}{\log x} +it} \log N(n)}
\frac{\log (x/N(n))}{\log x}}+( \log T+\frac {\log N(m)}{2} )(\half-\sigma+\frac{1+\lam}{\log x})+ O\Big( \frac{1}{\log x}\Big),
\end{align}
where the sum $\displaystyle \sum_{(n)}$ means that the sum is over integral ideals of $\mathcal{O}_K$.
\end{lemma}
\begin{proof}
  We denote $s$ for $\sigma+it$ and we interpret $\log |L(s, \chi)|$ as $-\infty$ when $L(s, \chi)=0$. Thus we may suppose $L(\sigma+it, \chi) \neq 0$ in the rest of the proof.  Recall the associated function $\Lambda(s, \chi)$ defined as in \eqref{Lambda}. Here $\Lambda(s, \chi)$ is analytical in the entire complex plane since $\chi$ is non-principal. As $\Gamma(s)$ has simple poles at the non-positive rational integers (see \cite[\S 10]{Da}), we see from the expression of $\Lambda(s, \chi)$ from \eqref{Lambda} that $L(s, \chi)$ has simple zeros at $s=0, -1, -2, \ldots$, these are called the trivial zeros of $L(s, \chi)$. Since we are assuming GRH, we know that the non-trivial zeros of $L(s, \chi)$ are precisely the zeros of $\Lambda(s, \chi)$.

  Let $\rho=\tfrac 12+i\gamma$ run over the non-trivial zeros of $L(s, \chi)$. We then deduce from \cite[Theorem 5.6]{HIEK} and the observation that $L(s,\chi)$ is analytic at $s=1$ since $\chi$ is non-principal and primitive that
\begin{align}
\label{Lproductzeros}
  \Lambda(s,\chi)=e^{A + Bs}\prod_{\rho}\left(1 - \frac{s}{\rho}\right)e^{\frac{s}{\rho}},
\end{align}
where $A, B=B(\chi)$ are constants.

    Taking the logarithmic derivative on both sides of \eqref{Lproductzeros} and making use of \eqref{Lambda}, we obtain that
\begin{equation*}
-\frac{L'}{L}(s,\chi)=\frac{\Gamma'}{\Gamma}(s) - \log\pi + \frac{1}{2}\log N(m) - B - \sum_{\rho}\left(\frac{1}{s-\rho} + \frac{1}{\rho}\right).
\end{equation*}
    This implies that
\begin{align}
\label{Lderivreal}
-\Re{\frac{L'}{L}(s,\chi)}=\Re{\frac{\Gamma'}{\Gamma}(s)} - \log\pi + \frac{1}{2}\log N(m) - \Re(B) - \sum_{\rho}\Re{\left(\frac{1}{s-\rho} + \frac{1}{\rho}\right)}.
\end{align}

    On the other hand, combining the functional equation \eqref{fneqn} with \eqref{Lproductzeros}, we see that
\begin{align*}
  e^{A + Bs}\prod_{\rho}\left(1 - \frac{s}{\rho}\right)e^{\frac{s}{\rho}}= W(\chi)(N(m))^{-1/2}e^{A + B(1-s)}\prod_{\rho}\left(1 - \frac{1-s}{\rho}\right)e^{\frac{1-s}{\rho}}.
\end{align*}
  Taking logarithmic derivative on both sides of the above expression, we obtain that
\begin{equation*}
  2B=-\sum_{\rho}\left(\frac{1}{s-\rho} + \frac{1}{1-s-\rho} + \frac{1}{\rho} + \frac{1}{\rho}\right)=-2\sum_{\rho}\frac{1}{\rho}.
\end{equation*}
   Here the second equality above follows by noting that the terms containing $1-s-\rho$ and $s-\rho$ cancel as both $\rho, 1-\rho$ are zeros from the functional equation \eqref{fneqn}.
We note here that
\begin{equation*}
   \sum_{\rho}\frac{1}{\rho}
\end{equation*}
   is convergent since if $\rho$ is a zero, so is $\overline{\rho}$ and we have
\begin{align*}
  \rho^{-1} + \overline{\rho}^{-1}\ll|\rho|^{-2}
\end{align*}
   and we know that $\displaystyle \sum_{\rho}\frac{1}{|\rho|^2}$ is convergent by \cite[Lemma 5.5]{iwakow}.

   We then deduce that
\begin{equation*}
\Re(B)=-\sum_\rho\Re(\rho^{-1}).
\end{equation*}
   Combining the above with \eqref{Lderivreal}, we see that
\begin{align}
\label{LprimeLbound}
\begin{split}
-\Re{\frac{L'}{L}(s,\chi)}&=\Re{\frac{\Gamma'}{\Gamma}(s)} - \log\pi + \frac{1}{2}\log N(m) - \Re{\sum_{\rho}\left(\frac{1}{s-\rho}\right)}\\
&=\Re{\frac{\Gamma'}{\Gamma}(s)} - \log\pi + \frac{1}{2}\log N(m)-F(s) \\
&=  \log (|t|+1)+\frac{1}{2}\log N(m) +O(1) - F(s) \\
& \leq  \log T+\frac{1}{2}\log N(m) +O(1) - F(s).
\end{split}
\end{align}
  where the third line above follows from $\frac{\Gamma'}{\Gamma}(s)=\log s+O(|s|^{-1})$ by (6) of \cite[\S 10]{Da} and where we define
\begin{align*}
 F(s) = \Re{\sum_{\rho} \frac{1}{s-\rho}}= \sum_{\rho} \frac{\sigma-1/2}{(\sigma-1/2)^2+
 (t-\gamma)^2}.
\end{align*}

  Integrating the last expression given for $-\Re{\frac{L'}{L}(s,\chi)}$ in \eqref{LprimeLbound} from $\sigma = \Re(s)$ to $\sigma_0 > \frac 12$, we obtain by setting $s_0 =\sigma_0 +it$ that
\begin{align}
\label{Lprimediffbound}
\begin{split}
& \log |L(s, \chi)| - \log |L(s_0, \chi)| \\
& = \Big(\log T+\frac {\log N(m)}{2} +O(1)\Big) (\sigma_0-\sigma) -\int_{\sigma}^{\sigma_0} F(u+it) du \\
& \leq  (\sigma_0-\sigma) \Big(\log T+\frac {\log N(m)}{2} +O(1)\Big),
\end{split}
\end{align}
  where the last inequality above follows from the observation that $F(s) \geq 0$ for all $s$ satisfying $\Re(s) \geq \half$.

  Next, we deduce upon integrating term by term using the Dirichlet series expansion of
 $-\frac{L^{\prime}}{L}(s+w, \chi)$ that
$$
 \frac{1}{2\pi i} \int\limits_{(c)} -\frac{L^{\prime}}{L}(s+w, \chi)
 \frac{x^w}{w^2} dw  = \sum_{\substack{(n) \\ N(n) \le x }}
  \frac{\Lambda_{[i]}(n)\chi(n)}{N(n)^s} \log \leg {x}{N(n)},
 $$
   where $c>2$ is a large number. Now moving the line of integration in the above expression to the left and calculating residues, we see also that
 $$
 \frac{1}{2\pi i} \int\limits_{(c)} -\frac{L^{\prime}}{L}(s+w, \chi)
 \frac{x^w}{w^2} dw =-\frac{L^{\prime}}{L}(s, \chi) \log x - \Big(\frac{L^{\prime}}{L}(s, \chi)\Big)^{\prime}
 -\sum_{\rho} \frac{x^{\rho-s}}{(\rho-s)^2} -\sum_{k=0}^{\infty}
 \frac{x^{-k-s}}{(k+s)^2}.
 $$

  Comparing the above two expressions, we deduce that unconditionally, for any $x \ge 2$, we have
 \begin{align}
\label{Lprimeseries}
\begin{split}
 -\frac{L^{\prime}}{L}(s, \chi)= \sum_{\substack{(n) \\ N(n) \le x }}
 \frac{\Lambda_{[i]}(n)}{N(n)^s} \frac{\log (x/N(n))}{\log x} &+ \frac{1}{\log x} \Big(\frac{L^{\prime}}{L}(s, \chi)\Big)^{\prime}
 + \frac{1}{\log x} \sum_{\rho} \frac{x^{\rho-s}}{(\rho-s)^2}
 + \frac{1}{\log x} \sum_{k=0}^{\infty} \frac{x^{-k-s}}{(k+s)^2}.
\end{split}
 \end{align}

 We now integrate the real parts on both sides of \eqref{Lprimeseries} over $\sigma=\Re(s)$ from $\sigma_0$ to $\infty$ to see that for $x \geq 2$,
 \begin{align}
\label{logL}
\begin{split}
\log |L(s_0, \chi)| = \Re \Big( \sum_{\substack{(n) \\ N(n) \le x}} \frac{\Lambda_{[i]}(n)}{N(n)^{s_0} \log N(n)} \frac{\log (x/N(n))}{\log x}
  &- \frac{1}{\log x} \frac{L^{\prime}}{L}(s_0, \chi)\\
  & + \frac{1}{\log x} \sum_{\rho}
 \int_{\sigma_0}^{\infty} \frac{x^{\rho-s}}{(\rho-s)^2} d\sigma +O\Big(\frac{1}{\log x}\Big)\Big).
\end{split}
\end{align}
 Observe that
$$
\sum_{\rho}\Big|\int_{\sigma_0}^{\infty} \frac{x^{\rho -s}}{(\rho -s)^2} d\sigma\Big|
\le \sum_{\rho}\int_{\sigma_0}^{\infty}\frac{ x^{\frac 12-\sigma}}{|s_0-\rho|^2} d\sigma
= \sum_{\rho}\frac{x^{\frac 12-\sigma_0}}{|s_0-\rho|^2 \log x}= \frac{x^{\frac 12-\sigma_0}F(s_0)}{(\sigma_0-\frac 12)\log x}.
$$
 Applying this and \eqref{LprimeLbound} in \eqref{logL}, we see that
 \begin{align}
\label{logLbound}
\begin{split}
 \log |L(s_0, \chi)|
 &\le  \Re \sum_{\substack{(n) \\ N(n)\le x}} \frac{\Lambda_{[i]}(n)}{N(n)^{s_0} \log N(n)} \frac{\log (x/N(n))}{\log x}
 + \frac{1}{\log x} ( \log T+\frac {\log N(m)}{2} ) \\
& \hskip .5 in+F(s_0) \Big( \frac{x^{\frac 12-\sigma_0}}{(\sigma_0-\half) \log^2 x} -\frac{1}{\log x}
 \Big)
 + O\Big(\frac{1}{\log x}\Big).
\end{split}
 \end{align}

Adding inequalities \eqref{Lprimediffbound} and \eqref{logLbound}, we deduce that
\begin{align*}
 \log |L(s, \chi)|  &\le
 \Re  \sum_{\substack{(n) \\ N(n)\le x}} \frac{\Lambda_{[i]}(n)}{N(n)^{s_0} \log N(n)} \frac{\log (x/N(n))}{\log x}
 + ( \log T+\frac {\log N(m)}{2} ) \Big(\sigma_0 -\sigma + \frac{1}{\log x}\Big)
 \\
 &\hskip .5 in +F(s_0) \Big( \frac{x^{\frac 12-\sigma_0}}{(\sigma_0-\half) \log^2 x} -\frac{1}{\log x}
 \Big) + O\Big(\frac{1}{\log x}\Big)+O(\sigma_0-\sigma).
 \end{align*}
  The assertion of the proposition now follows by setting $\sigma_0 =\frac12 + \frac{\lam}{\log x}$ with $\lam \ge \lam_0$ and by omitting the term involving $F(s_0)$ since it is negative.
 \end{proof}

    Our next lemma treats essentially the sum over prime squares in \eqref{logLupperbound}.
\begin{lemma}
\label{lem: primesquareLbound}
Assume GRH for $\zeta_K(s)$. Let $\lambda_0$ be given as in Lemma \ref{lem: logLbound} and let $z \in \mathbb{C}$ with $0 \leq \Re (z) \leq \frac{1}{\log X}$.
We have for $x \geq 10$,
\begin{align}
\label{primesquarebound}
 \sum_{\substack{ \varpi \equiv 1 \bmod {(1+i)^3} \\ N(\varpi) \leq x^{1/2} }}  \frac{1}{N(\varpi)^{1+\frac{2\lambda_0}{\log x}+2z}}  \frac{\log (\frac{x}{N(\varpi)^2})}{\log x} = \mathcal{L}(z,x)+O(1),
\end{align}
where $\mathcal{L}(z,x)$ is given in \eqref{Ldef}.
\end{lemma}
\begin{proof}
   We first note that under GRH for $\zeta_K(s)$, the prime ideal theorem has the following form (see \cite[Theorem 5.15]{iwakow})
\begin{align}
\label{PIT}
 \sum_{\substack{ \varpi \equiv 1 \bmod {(1+i)^3} \\ N(\varpi) \leq y }}\log N(\varpi) = y+O(\sqrt{y} \left(\log Xy)^2 \right).
\end{align}

   It follows from this and $\Re (z)\geq 0$ that we have
 \begin{align*}
  \sum_{\substack{ \varpi \equiv 1 \bmod {(1+i)^3} \\ N(\varpi) \leq x^{1/2} }}
    \frac{1}{N(\varpi)^{1+\frac{2\lambda_0}{\log x}+2z}}
    \frac{\log N(\varpi)}{\log x} \ll \sum_{\substack{ \varpi \equiv 1 \bmod {(1+i)^3} \\ N(\varpi)\leq \sqrt{x}}}
    \frac{1}{N(\varpi)}
    \frac{\log N(\varpi)}{\log x} =O(1).
 \end{align*}
   We then deduce from this that it suffices to establish \eqref{primesquarebound} with the left-hand side expression in \eqref{primesquarebound} being replaced by $f(\frac{2\lambda_0}{\log x}+2z,x)$, where
\begin{align*}
  f(z,x)= \sum_{\substack{ \varpi \equiv 1 \bmod {(1+i)^3} \\ N(\varpi) \leq x^{1/2} }} \frac{1}{N(\varpi)^{1+z}}
\end{align*}

  It is easily seen that $f(\frac{2\lambda_0}{\log x}+2z,x)=O(1)$ when $|z| \geq 1$ using \eqref{PIT} and partial summation. The same procedure also implies that
\begin{align}
\label{sumprimerecip}
   \sum_{\substack{ \varpi \equiv 1 \bmod {(1+i)^3} \\ N(\varpi) \leq x^{1/2} }}
    \frac{1}{ N(\varpi)}=\log \log x+O(1).
\end{align}

   It follows that when $|z| \leq (\log x)^{-1}$, we have
\begin{align*}
  f(\frac{2\lambda_0}{\log x}+2z,x)=&\sum_{\substack{ \varpi \equiv 1 \bmod {(1+i)^3} \\ N(\varpi) \leq x^{1/2} }}
    \frac{1}{ N(\varpi)}+\sum_{\substack{ \varpi \equiv 1 \bmod {(1+i)^3} \\ N(\varpi) \leq x^{1/2} }}
    (\frac{1}{ N(\varpi)^{1+\frac{2\lambda_0}{\log x}+2z}}-\frac 1{N(\varpi)}) \\
=& \log \log x+\sum_{\substack{ \varpi \equiv 1 \bmod {(1+i)^3} \\ N(\varpi) \leq x^{1/2} }}\frac {\log N(\varpi)}{N(\varpi)} \int^{\frac{2\lambda_0}{\log x}+2z}_0 N(\varpi)^{u}du +O(1) \\
=& \log \log x+O(\frac 1{\log x}\sum_{\substack{ \varpi \equiv 1 \bmod {(1+i)^3} \\ N(\varpi) \leq x^{1/2} }}\frac {\log N(\varpi)}{N(\varpi)})  +O(1) \\
=& \log \log x+O(1),
 \end{align*}
  where the integral is along the line segment connecting the origin and the point $\frac{2\lambda_0}{\log x}+2z$ on the complex plane.

  In the remaining case when $(\log x)^{-1} \leq |z| \leq 1$, we note that by \eqref{PIT} and partial summation,
\begin{align*}
  \frac {\partial f}{\partial z}=-\sum_{\substack{ \varpi \equiv 1 \bmod {(1+i)^3} \\ N(\varpi) \leq x^{1/2} }} \frac{\log N(\varpi)}{N(\varpi)^{1+z}}=-\int^{\sqrt{x}}_1\frac 1{u^{1+z}}d(u+O(u^{1/2}(\log u)^2))
=-\frac {1}{z}+\frac {x^{-z/2}}{z}+O(1).
\end{align*}
   Now we assume that $0 \le \Re(w)$ and $|w| \geq (\log x)^{-1}$ and we integrate $\partial f/\partial z$ from $1$ to $w$ to see that
\begin{align*}
  f(w,x)=& -\log w+\int^w_1\frac {x^{-u/2}}{u}du+O(1)=-\log |w|-\frac { 2x^{-u/2}}{u \log x} \Big |^w_1-
2\int^w_1\frac {x^{-u/2}}{u^2 \log x}du +O(1) \\
=& -\log |w|-
2\int^w_1\frac {x^{-u/2}}{u^2 \log x}du +O(1).
\end{align*}
   We break the integration $\int^w_1\frac {x^{-u/2}}{u^2 \log x}du$ into two parts, one horizontal integration along the $x$-axis from $1$ to $\Re(w) \gg (\log x)^{-1}$, and the other vertical integration from $\Re(w)$ to $z$. The horizontal integration is
\begin{align*}
 \ll \int^{\infty}_{(\log x)^{-1}} \frac {1}{u^2 \log x}du=O(1).
\end{align*}
  If we write $w=\sigma+it$, then the vertical integration can be evaluated by breaking the integral over $t$ for $|t| \leq \Re(z)$ and $|t| > \Re(z)$. We obtain this way that the vertical integration is
\begin{align*}
 \ll \int^{\Re(z)}_{-\Re(z)} \frac {1}{\Re(z)^2 \log x}dt+ \int^{\infty}_{\Re(z)} \frac {1}{t^2 \log x}dt \ll \frac {1}{\Re(z) \log x} =O(1).
\end{align*}
  It follows that we have $f(w,x)=-\log |w|+O(1)$ when $0 \le \Re(w)$ and $|w| \geq (\log x)^{-1}$. In particular, this applies to the case when $w=\frac{2\lambda_0}{\log x}+2z$, thus completes the proof.
\end{proof}

   Lastly, we present a mean value estimation which will be applied to estimate the sum over primes in \eqref{logLupperbound} in our proof of Theorem \ref{upperbound}.
\begin{lemma}
\label{lem:2.5}
 Let $X$ and $y$ be real numbers. For fixed $0<\varepsilon<1$, let $k$ be a natural number with $y^k\leq X^{1/2-\varepsilon}$. Then for any complex numbers $a(\varpi)$, we have
 \begin{align*}
  \sumstar_{\substack{(d,2)=1 \\ N(d)\leq X}}\left|\sum_{\substack{\varpi \equiv 1 \bmod {(1+i)^3} \\ N(\varpi) \leq y}}\frac{a(\varpi)\chi_{(1+i)^5d}(\varpi)}{N(\varpi)^{\frac{1}{2}}}\right|^{2k}
  \ll_{\varepsilon} X\frac{(2k)!}{k!2^k}\left(\sum_{\substack{\varpi \equiv 1 \bmod {(1+i)^3} \\ N(\varpi) \leq y}} \frac{|a(\varpi)|^2}{N(\varpi)}\right)^k.
 \end{align*}
\end{lemma}
\begin{proof}
   Let $W(t)$ be any non-negative smooth function that is supported on $(\frac 1{2}-\varepsilon_1, 1+\varepsilon_1)$ for some fixed small $0<\varepsilon_1<1/2$ such that $W(t) \gg 1$ for $t \in (\half, 1)$. We have that
\begin{align*}
 & \sumstar_{\substack{(d,2)=1 \\ N(d)\leq X}}\left|\sum_{\substack{\varpi \equiv 1 \bmod {(1+i)^3} \\ N(\varpi) \leq y}}\frac{a(\varpi)\chi_{(1+i)^5d}(\varpi)}{N(\varpi)^{\frac{1}{2}}}\right|^{2k} \ll \sum_{\substack{(d,2)=1}}\left|\sum_{\substack{\varpi \equiv 1 \bmod {(1+i)^3} \\ N(\varpi) \leq y}}\frac{a(\varpi)\chi_{(1+i)^5d}(\varpi)}{N(\varpi)^{\frac{1}{2}}}\right|^{2k}W(\frac {N(d)}{X}) \\
\ll & \sum_{\substack{(d,2)=1}}  \Big| \sum_{\substack{\varpi_1, \dots, \varpi_k \equiv 1 \bmod {(1+i)^3} \\ N(\varpi_1), \dots, N(\varpi_k) \leq y}} \frac{a(\varpi_1) \dots a(\varpi_k) \ }{\sqrt{N(\varpi_1 \dots \varpi_k)}}
\leg{(1+i)^5d}{\varpi_1\cdots \varpi_k} \Big|^2W(\frac {N(d)}{X}).
\end{align*}
  We further expand out the square in the last sum above and treat the sum over $d$ by applying \eqref{dchar} first and then using the smoothed version of P{\' o}lya-Vinogradov inequality \eqref{PV} for number fields when the product of the primes involved is not a perfect square to see that we have
\begin{align*}
 & \sum_{\substack{(d,2)=1}}  \Big| \sum_{\substack{\varpi_1, \dots, \varpi_k  \equiv 1 \bmod {(1+i)^3} \\ N(\varpi_1), \dots, N(\varpi_k) \leq y}} \frac{a(\varpi_1) \dots a(\varpi_k) \ }{\sqrt{N(\varpi_1 \dots \varpi_k)}}
\leg{(1+i)^5d}{\varpi_1\cdots \varpi_k} \Big|^2W(\frac {N(d)}{X}) \\
\ll &
X \sum_{\substack{\varpi_1, \dots, \varpi_{2k} \equiv 1 \bmod {(1+i)^3} \\ N(\varpi_1), \dots, N(\varpi_{2k}) \leq y \\ \varpi_1 \dots \varpi_{2k} = \square}} \frac{|a(\varpi_1) \dots a(\varpi_{2k}) |}{\sqrt{N(\varpi_1 \dots \varpi_{2k})}} + O\left(\sum_{\substack{\varpi_1, \dots, \varpi_{2k} \equiv 1 \bmod {(1+i)^3} \\ N(\varpi_1), \dots, N(\varpi_{2k}) \leq y}}|a(\varpi_1) \dots a(\varpi_{2k})|y^{2k\varepsilon'}   \right),
\end{align*}
  where we write $\square$ for a square of an element in $\mathcal{O}_K$.

We take $\varepsilon'$ small enough so that $y^{2k}y^{2k\varepsilon'} \ll X$. Then an argument similar to that in the proof of \cite[Lemma 6.3]{S&Y} leads to the  assertion of the lemma.
\end{proof}

\subsection{Completion of the proof}

   With lemmas \ref{lem: logLbound}-\ref{lem:2.5} now available, we proceed to establish an upper bound for the frequency of large values of $\log|L(\tfrac{1}{2}+z_1,\chi_{(1+i)^5d})L(\tfrac{1}{2}+z_2,\chi_{(1+i)^5d})|$.
\begin{prop}
\label{propNbound}
Assume GRH for $\zeta_K(s)$ and $L(s,\chi_{(1+i)^5d})$ for all odd, square-free $d \in \mathcal{O}_K$. Let $X$ be large and let $z_1,z_2 \in \mathbb{C}$ with $0 \leq \Re (z_1), \Re (z_2) \leq \frac{1}{\log X}$, and $|\Im (z_1)|, |\Im (z_2)| \leq X$. Let $\mathcal{N}(V;z_1,z_2,X)$ denote the number of
 odd, square-free $d \in \mathcal{O}_K$ such that $N(d)\leq X$ and
 \begin{align*}
  \log|L(\tfrac{1}{2}+z_1,\chi_{(1+i)^5d})L(\tfrac{1}{2}+z_2,\chi_{(1+i)^5d})|\geq
  V +\mathcal{M}(z_1,z_2,X).
 \end{align*}
 Then for $10\sqrt{\operatorname{log}\operatorname{log} X}\leq V \leq \mathcal{V}(z_1,z_2,X)$,  we have
 \begin{align*}
  \mathcal{N}(V;z_1,z_2,X)\ll X\operatorname{exp}\left(-\frac{V^2}{2\mathcal{V}(z_1,z_2,X)}
  \left(1-\frac{25}{\operatorname{log}\operatorname{log}\operatorname{log}X}\right)\right);
 \end{align*}
 for $\mathcal{V}(z_1,z_2,X)< V \leq \frac{1}{16}\mathcal{V}(z_1,z_2,X)\operatorname{log}\operatorname{log}\operatorname{log}X$,
 we have
 \begin{align*}
  \mathcal{N}(V;z_1,z_2,X)\ll X\operatorname{exp}\left(-\frac{V^2}{2\mathcal{V}(z_1,z_2,X)}
  \left(1-\frac{15V}{\mathcal{V}(z_1,z_2,X)\operatorname{log}\operatorname{log}\operatorname{log}X}\right)^2\right);
 \end{align*}
 for $\frac{1}{16}\mathcal{V}(z_1,z_2,X)\operatorname{log}\operatorname{log}
 \operatorname{log}X<V$, we have
 \begin{align*}
  \mathcal{N}(V;z_1,z_2,X)\ll X\operatorname{exp}\left(-\frac{1}{1025}V\operatorname{log}V\right).
 \end{align*}
\end{prop}
\begin{proof}
Let $\lambda_0$ be the constant defined in Lemma \ref{lem: logLbound} and apply this Lemma with $\lambda=\lambda_0+\Re(z_i)\log x, i=1,2$ and $T=X$ to $L(\tfrac{1}{2}+z_i,\chi_{(1+i)^5d})$ for $N(d) \leq X$, we see that for $2 \leq x \leq X$,
\begin{align*}
\log |L(\tfrac{1}{2}+z_i,\chi_{(1+i)^5d})| & \leq \Re \left( \sum_{\substack{n \equiv 1 \bmod {(1+i)^3} \\ N(n) \leq x}} \frac{ \Lambda_{[i]} (n) \chi_{(1+i)^5d} (n)}{N(n)^{\frac{1}{2}+\frac{\lambda_0}{\log x}+z_i} \log N(n)} \frac{\log (\frac{x}{N(n)})}{\log x} \right)+ \frac 3{2}(1+\lambda_0)\frac{\log X}{\log x} +O \left (\frac{1}{\log x} \right ), \  i=1,2.
\end{align*}
  We then deduce that
\begin{align}
&\log  |L(\tfrac{1}{2}+z_1 , \chi_{(1+i)^5d})||L(\tfrac{1}{2}+z_2,\chi_{(1+i)^5d})|  \nonumber \\
 &\leq \Re \left( \sum_{\substack{ \varpi \equiv 1 \bmod {(1+i)^3} \\ N(\varpi^l) \leq x \\ l \geq 1}} \frac{\chi_{(1+i)^5d} (\varpi^l)}{l N(\varpi)^{l(\tfrac{1}{2}+\frac{\lambda_0}{\log x})}}  (N(\varpi)^{-lz_1} + N(\varpi)^{-lz_2}) \frac{\log (\frac{x}{N(\varpi)^l})}{\log x} \right)+ 3 (1+\lambda_0)\frac{\log X}{\log x} +O \left (\frac{1}{\log x} \right ).
\label{equ:3.1}
\end{align}

The terms with $l \geq 3$ in the the above sum contribute $O(1)$. Using the fact $\displaystyle \sum_{\varpi |d} \frac{1}{N(\varpi)} \ll \log\log\log N(d)$,  we deduce from Lemma \ref{lem: primesquareLbound} that
\begin{align*}
\begin{split}
&\Re  \left( \sum_{\substack{ \varpi \equiv 1 \bmod {(1+i)^3}\\ N(\varpi) \leq x^{1/2} }}  \frac{\chi_{(1+i)^5d}(\varpi^2)}{2 N(\varpi)^{1+\frac{2\lambda_0}{\log x}}}  (N(\varpi)^{-2z_1} + N(\varpi)^{-2z_2}) \frac{\log (\frac{x}{N(\varpi)^2})}{\log x} \right) \\
=& \Re  \left( \sum_{\substack{ \varpi \equiv 1 \bmod {(1+i)^3} \\ N(\varpi) \leq x^{1/2} }}  \frac{1}{2 N(\varpi)^{1+\frac{2\lambda_0}{\log x}}}  (N(\varpi)^{-2z_1} + N(\varpi)^{-2z_2}) \frac{\log (\frac{x}{N(\varpi)^2})}{\log x} \right) \\
=&  \mathcal{M}(z_1,z_2,x)+ O(\log\log\log X) \\
\leq &  \mathcal{M}(z_1,z_2,X)+ O(\log\log\log X),
\end{split}
\end{align*}
  where $\mathcal{M}(z_1,z_2,x)$ is defined as in \eqref{Mdef}.

 Applying the above estimation in  \eqref{equ:3.1}, we obtain  that
 \begin{align}
\label{equ:3.3}
\begin{split}
 & \log  |L(\tfrac{1}{2}+z_1 , \chi_{(1+i)^5d})||L(\tfrac{1}{2}+z_2,\chi_{(1+i)^5d})| \\
    \leq & \Re \left( \sum_{\substack{ \varpi \equiv 1 \bmod {(1+i)^3} \\ N(\varpi) \leq x}} \frac{\chi_{(1+i)^5d} (\varpi)}{N(\varpi)^{\frac{1}{2}+\frac{\lambda_0}{\log x}}}  (N(\varpi)^{-z_1} +   N(\varpi)^{-z_2}) \frac{\log (\frac{x}{N(\varpi)})}{\log x} \right)
    +\mathcal{M}(z_1,z_2,X)+
    \frac{5 \log X}{\log x} + O(\log \log \log X).
\end{split}
 \end{align}
   By taking $x = \log X$ in (\ref{equ:3.3}) and bounding the sum  over $\varpi$ in (\ref{equ:3.3}) trivially (with the help of \eqref{PIT}), we see that $\mathcal{N} (V; z_1,z_2,X) =0$  for $V > \frac{6\log X}{\log \log X}  $. Thus, we can assume $V \leq   \frac{6\log X}{\log \log X}$.

  In what follows, we shall denote $\mathcal{V}$ for $\mathcal{V} (z_1,z_2,X)$ defined in \eqref{Vdef} and we note that $\log\log X + O(1) \leq \mathcal{V}(z_1,z_2,x) \leq 4\log\log X+O(1)$. We now set $x=X^{A/V}$ with
\begin{align*}
 A =\left\{
 \begin{array}
  [c]{ll}
  \frac{1}{2}\log\log\log X & 10 \sqrt{\log \log X}\leq V\leq \mathcal{V},\\
  \frac{\mathcal{V}}{2V}\log\log\log X & \mathcal{V}<V\leq \frac{1}{16}\mathcal{V}\log\log\log X,\\
  8 & V > \frac{1}{16}\mathcal{V}\log\log\log X.
 \end{array}
 \right.
\end{align*}

  We further denote $z=x^{1/\log \log X}$, $M_1$ for the real part of the sum in \eqref{equ:3.3} truncated to $N(\varpi) \leq z$,  $M_2$ for the real part of the sum in \eqref{equ:3.3} over $z < N(\varpi) \leq x$. We then deduce that
\[
 \log  |L(\tfrac{1}{2}+z_1 , \chi_{(1+i)^5d})||L(\tfrac{1}{2}+z_2,\chi_{(1+i)^5d})|
    \leq M_1 + M_2 + \mathcal{M}(z_1,z_2,X) + \frac{5V}{A}.
\]
It follows from this that if  $\log  |L(\tfrac{1}{2}+z_1 , \chi_{(1+i)^5d})||L(\tfrac{1}{2}+z_2,\chi_{(1+i)^5d})| \geq V +  \mathcal{M} (z_1,z_2,X)$, then we have either
\[
M_2 \geq \tfrac{V}{A}, \text{ or } M_1 \geq V_1 := V(1-\tfrac{6}{A}).
\]

Now, we define
\begin{align*}
\operatorname{meas}(X;M_1) & = \# \{N(d) \leq X \ : \ d \text{ odd and square-free, } M_1 \geq V_1 \},\\
\operatorname{meas}(X;M_2) & = \# \{N(d) \leq X \ : \ d \text{ odd and square-free, } M_2 \geq \tfrac{V}{A}  \}.
\end{align*}

We then take $m = [0.9(\frac{V}{2A})]$ to see that by Lemma \ref{lem:2.5}, we have
\begin{align*}
 (V/A)^{2m} \operatorname{meas}(X;S_2) \leq  \sumstar_{\substack{(d,2)=1 \\ N(d)\leq X}}|M_2|^{2m} \ll X\frac{(2m)!}{m!2^m}\left(\sum_{\substack{\varpi \equiv 1 \bmod{(1+i)^3} \\ z<N(\varpi) \leq x}}
\frac{4}{N(\varpi)}\right)^m \ll X (3m \log \log \log X)^m,
\end{align*}
  where the last estimation above follows from \eqref{sumprimerecip} and Stirling's formula (see \cite[(5.112)]{iwakow}), which implies that
\begin{align*}
 \frac{(2m)!}{m!2^m} \ll (\frac {2m}{e})^m.
\end{align*}

   We then deduce that
\begin{equation}
\label{equ:bd-S-2}
\operatorname{meas}(X;M_2) \ll X\operatorname{exp}\left(-\frac{V}{3A}\log V\right).
\end{equation}

Next, we estimate $\operatorname{meas}(X;M_1)$. For any $m \leq \frac{(\frac{1}{2}-\varepsilon)\log X}{ \log z}$, we obtain using Lemma \ref{lem:2.5} that
\begin{align}
\label{equ:3.4}
V^{2m}_1\operatorname{meas}(X;M_1) \leq  \sumstar_{\substack{(d,2)=1 \\ N(d)\leq X}}|M_1|^{2m} \ll X\frac{(2m)!}{m!2^m}\left(\sum_{\substack{\varpi \equiv 1 \bmod{(1+i)^3} \\ N(\varpi) \leq z}}\frac{|a(\varpi)|^2}{N(\varpi)}\right)^m,
\end{align}
where
\begin{align*}
 a(\varpi)=\frac{\Re(N(\varpi)^{-z_1}+N(\varpi)^{-z_2})\log (\frac{x}{N(\varpi)})}{N(\varpi)^{\frac{\lambda_0}{\log x}}\log x}.
\end{align*}
By arguing as in the proof of Lemma \ref{lem: primesquareLbound}, we see that
\begin{align*}
\sum_{\substack{\varpi \equiv 1 \bmod{(1+i)^3} \\ N(\varpi) \leq z}}\frac{|a(\varpi)|^2}{N(\varpi)}
\leq \frac{1}{4}\sum_{\substack{\varpi \equiv 1 \bmod{(1+i)^3} \\ N(\varpi) \leq z}}\frac{1}{N(\varpi)}(N(\varpi)^{-z_1}+N(\varpi)^{-\overline{z_1}}+N(\varpi)^{-z_2}+N(\varpi)^{-\overline{z_2}})^2
=\mathcal{V}+O(1).
\end{align*}
Combining with (\ref{equ:3.4}), this implies that
\begin{align*}
\operatorname{meas}(X ; M_1)
\ll X V_1^{-2m} \frac{(2m)!}{m!2^m} (\mathcal{V} + O(1))^m
\ll X\left(\frac{2m}{e} \cdot \frac{\mathcal{V}+O(1)}{V_1^2}\right)^m.
\end{align*}

We now take $m = [\frac{V_1^2} { 2\mathcal{V}}]$ when $V \leq (\log \log X)^2$ and $ m = [10V]$ otherwise to see that in either case, we have for $X$ large,
\begin{align*}
m \leq \frac{(\frac{1}{2}-0.1)\log X}{ \log z}.
\end{align*}

  A little calculation then shows that
\begin{align*}
\operatorname{meas}(X ; M_1) \ll  X\operatorname{exp}\left(-\frac{V_1^2}{2\mathcal{V}}\left(1+O\left(\frac{1}{\log\log X}\right)\right)\right)
+X\operatorname{exp}\left(-V\log V \right).
\end{align*}

  We then deduce from the above and \eqref{equ:bd-S-2} that
\begin{align*}
\mathcal{N}(V;z_1,z_2,X) \ll X\operatorname{exp}\left(-\frac{V_1^2}{2\mathcal{V}}\left(1+O\left(\frac{1}{\log\log X}\right)\right)\right)
+X\operatorname{exp}\left(-V\log V\right)+X\operatorname{exp}\left(-\frac{V}{3A}\log V\right).
\end{align*}

  It is then easy to check that this leads to the assertion of the proposition.
\end{proof}

  We now return to the proof of Theorem \ref{upperbound} by noting that for $k$ given as in the theorem, Proposition \ref{propNbound} implies that for all $V \geq 10\sqrt{\log \log X}$,
\begin{align}
 \label{equ:rough-01}
  \mathcal{N}(V;z_1,z_2,X)\ll
\begin{cases}
  X, \quad V < 10\sqrt{\log \log X}, \\
  X(\operatorname{log}X)^{o(1)}\operatorname{exp}\left(-\frac{V^2}{2\mathcal{V}(z_1,z_2,X)}\right), \quad 10\sqrt{\log \log X} \leq V \leq 4k \mathcal{V}(z_1,z_2,X), \\
   X(\operatorname{log}X)^{o(1)}\operatorname{exp}(-4kV), \quad  V > 4k \mathcal{V}(z_1,z_2,X).
\end{cases}
\end{align}
  Note further that we have
\begin{align*}
  \sumstar_{\substack{(d,2)=1 \\ N(d)\leq X}} |L(\tfrac{1}{2}+z_1 , \chi_{(1+i)^5d})||L(\tfrac{1}{2}+z_2,\chi_{(1+i)^5d})|^k =& -\int_{-\infty}^{+\infty}\operatorname{exp}(kV+k\mathcal{M}(z_1,z_2,X))d \mathcal{N}(V;z_1,z_2,X) \\
  =& k\int_{-\infty}^{+\infty}\operatorname{exp}(kV+k\mathcal{M}(z_1,z_2,X))\mathcal{N}(V;z_1,z_2,X)dV.
\end{align*}

  Now the assertion of Theorem \ref{upperbound} follows by supplying the bound given in \eqref{equ:rough-01} to evaluate the integration above.

\section{Proof of Theorems \ref{main-thm} and \ref{theo:mainthm} }
\label{sec: pfmainthm}

\subsection{Initial treatment}

  Let $\Phi$ be a smooth Schwartz class function  which is compactly supported on $[\frac{1}{2}, \frac{5}{2}]$ satisfying $0 \leq  \Phi(t) \leq 1$ for all $t$. We apply the  approximate functional equation \eqref{fcneqnL} to see that
\begin{align*}
 \sumstar_{(d,2)=1} L(\thalf, \chi_{(1+i)^5d})^4 \Phi\leg{N(d)}{X}
=\sumstar_{(d,2)=1}A_{N(d)}(d)^2\Phi\leg{N(d)}{X},
\end{align*}
where
\begin{align}
\label{equ:13.1}
A_{t}(d) = 2\sum_{\substack{n \equiv 1 \bmod {(1+i)^3}}} \frac{\chi_{(1+i)^5d}(n)d_{[i]}(n)}{N(n)^{\frac{1}{2}}} V
\left(\frac{ N(n)}{t} \right).
\end{align}
 For two parameters $U_1,U_2$ satisfying $X^{\frac{9}{10}} \leq U_1 \leq U_2 \leq X $, we define
\begin{align}
\label{equ:13.1+}
S(U_1,U_2)=\sumstar_{(d,2)=1} A_{U_1}(d) A_{U_2}(d) \Phi\leg{N(d)}{X}.
\end{align}

  We let
\begin{align}
\label{equ:def-h}
h(x,y,z) =  \Phi\leg{N(x)}{X}V
\left(\frac{ N(y)}{U_1} \right)V
\left(\frac{ N(z)}{U_2} \right).
\end{align}
  Then applying (\ref{equ:13.1}) to (\ref{equ:13.1+}) and using  the M\"obius inversion to remove the square-free condition in (\ref{equ:13.1+}), we obtain that
\begin{align*}
\begin{split}
S(U_1,U_2)=& 4\sideset{}{^*}\sum_{(d,2)=1}\sum_{\substack{n_1, n_2 \equiv 1 \bmod {(1+i)^3}}} \frac{\chi_{(1+i)^5d}(n_1n_2)d_{[i]}(n_1)d_{[i]}(n_2)}{N(n_1n_2)^{\frac{1}{2}}}h(d,n_1,n_2) \\
=& 4\sum_{(d,2)=1} \sum_{\substack{a \equiv 1 \bmod {(1+i)^3} \\ a^2|d}} \mu_{[i]}(a) \sum_{\substack{n_1, n_2 \equiv 1 \bmod {(1+i)^3}}} \frac{\chi_{(1+i)^5d}(n_1n_2)d_{[i]}(n_1)d_{[i]}(n_2)}{N(n_1n_2)^{\frac{1}{2}}}h(d,n_1,n_2) \\
=& 4\sum_{\substack{a \equiv 1 \bmod {(1+i)^3} }} \mu_{[i]}(a)  \sum_{(d,2)=1} \sum_{\substack{n_1, n_2 \equiv 1 \bmod {(1+i)^3}\\ (n_1n_2,a)=1}} \frac{\chi_{(1+i)^5d}(n_1n_2)d_{[i]}(n_1)d_{[i]}(n_2)}{N(n_1n_2)^{\frac{1}{2}}}h(a^2d,n_1,n_2).
\end{split}
\end{align*}
 Now we separate the terms with $N(a) \leq Y$ and with $N(a) > Y$ for some $Y \leq X$ to be chosen later, writing $S(U_1,U_2)=S_1 +S_2$, respectively.   We bound $S_2$ first in the following result. 
\begin{lemma}
\label{lem:13.1}
Unconditionally, we have  $S_2 \ll X^{1+\varepsilon}Y^{-1}$. Under GRH, we have $S_2 \ll   XY^{-1} (\log X)^{46}$.
\end{lemma}
\begin{proof}
We first write $d=lb^2$ with $l$ square-free and $b$ primary. We then let $c=ab$ and apply the definition of $h(x,y,z)$ in \eqref{equ:def-h} to see that
\begin{align}
\label{equ:13.4}
\begin{split}
S_2&= 4\sum_{\substack{c \equiv 1 \bmod {(1+i)^3}  }} \sum_{\substack{a \equiv 1 \bmod {(1+i)^3} \\ N(a) >  Y \\ a | c}}  \mu_{[i]}(a) \sideset{}{^*}\sum_{(l,2)=1}\sum_{\substack{n_1, n_2 \equiv 1 \bmod {(1+i)^3}\\ (n_1n_2,c)=1}} \frac{\chi_{(1+i)^5l}(n_1n_2)d_{[i]}(n_1)d_{[i]}(n_2)}{N(n_1n_2)^{\frac{1}{2}}}h (c^2l,n_1,n_2 )  \\
&= \frac{4}{(2\pi i )^2}\sum_{\substack{c \equiv 1 \bmod {(1+i)^3}  }} \sum_{\substack{a \equiv 1 \bmod {(1+i)^3} \\ N(a) >  Y \\ a | c}}  \mu_{[i]}(a)
\int\limits_{(\frac{1}{2}+\varepsilon)} \int\limits_{(\frac{1}{2}+\varepsilon)}   \frac{w(u)w(v)}{uv}  U_1^u U_2^v \\
&\quad \times \sideset{}{^*}\sum_{(l, 2)=1} \Phi\left( \frac{N(c^2l)}{X}\right)
 L_c (\tfrac{1}{2}+u,\chi_{(1+i)^5l})^2  L_c (\tfrac{1}{2}+v,\chi_{(1+i)^5l})^2 dudv,
\end{split}
\end{align}
 where we define $L_c(s,\chi)$ to be the function by removing the Euler factors from $L(s,\chi)$ at prime ideals dividing $c$.

Applying the estimation
\begin{align*}
& |L_c(\tfrac{1}{2}+u,\chi_{(1+i)^5l})^2L_c(\tfrac{1}{2}+v,\chi_{(1+i)^5l})^2| \ll |L_c(\tfrac{1}{2}+u,\chi_{(1+i)^5l})|^4+|L_c(\tfrac{1}{2}+u,\chi_{(1+i)^5l})|^4 \\
\ll & d_{[i]}^2(c)\left ( |L(\tfrac{1}{2}+u,\chi_{(1+i)^5l})|^4+|L(\tfrac{1}{2}+u,\chi_{(1+i)^5l})|^4 \right ),
\end{align*}
 we can bound $S_2$ by moving the lines of the integrations in \eqref{equ:13.4} to $\Re(u) = \Re(v) = \frac{1}{\log X}$ to see that
\begin{align}
\label{equ:13.5}
S_2 \ll (\log X )^2 \sum_{(c,2)=1} d_{[i]}(c)^4  \sum_{\substack{N(a)>Y\\ a|c}}  \int\limits_{(\frac{1}{\log X})} \int\limits_{(\frac{1}{\log X})}  |w(u)w(v)| \sideset{}{^*}\sum_{\substack{(l,2)=1 \\ N(l) \leq \frac{5X}{2N(c)^2} } }   |L (\tfrac{1}{2}+u,\chi_{(1+i)^5l} )|^4 \  |du| |dv|.
\end{align}

  By Corollary \ref{4momentupperbound}, we see that for $|\Im(u)| \leq \frac{5X(\log X)^2}{2N(c)^2}$,
\begin{align}
\label{equ:13.5+}
 \sumstar_{\substack{(l,2)=1 \\ N(l) \leq \frac{5X}{2N(c)^2}}} |L(\tfrac{1}{2}+u,\chi_{(1+i)^5l})| ^4
\ll \frac{X}{N(c)^2} (\log X)^{13}.
\end{align}

   Note that Lemma \ref{lem:2.3} implies that
\begin{align}
\label{equ:13.6'}
 \sumstar_{\substack{(l,2)=1 \\ N(l) \leq \frac{5X}{2N(c)^2}}} |L(\tfrac{1}{2}+u,\chi_{(1+i)^5l})| ^4
\ll \left( \frac{X}{N(c)^2} \right)^{1+\varepsilon} (1+|\Im(u)|^2)^{1+\varepsilon}.
\end{align}

   Applying \eqref{equ:13.5+} in \eqref{equ:13.5} when $|\Im(u)| \leq \frac{5X(\log X)^2}{2N(c)^2}$ and \eqref{equ:13.6'} in \eqref{equ:13.5} otherwise, together with the observation that $w(u)$ decreases exponentially in $\Im(u)$, we deduce that
\begin{align*}
S_2 \ll X (\log X)^{15} \sum_{(c,2)=1} \frac{d^4_{[i]}(c)}{N(c)^2}  \sum_{\substack{N(a)>Y\\ a|c}} 1
\ll X (\log X)^{15} \sum_{N(c)>Y}\frac{d^5_{[i]}(c)}{N(c)^2}  \ll XY^{-1} (\log X)^{46},
\end{align*}
 where the last estimation above follows from partial summation and the following estimation for $x>2$:
\begin{align}
\label{destimation}
\sum_{N(c) \leq x}d^5_{[i]}(c) \ll x(\log x)^{31}.
\end{align}

  The unconditionally estimation for $S_2$ is obtained similarly by applying \eqref{equ:13.6'} in \eqref{equ:13.5} for all $u$ and this
completes the proof of the lemma.
\end{proof}

 Next, we treat $S_1$ by applying the Poisson summation formula given in Lemma \ref{Poissonsumformodd} to recast it as
\begin{align}
\label{equ:13.a}
\begin{split}
S_1=& 4 \sum_{\substack{a \equiv 1 \bmod {(1+i)^3} \\ N(a) \leq Y }}  \mu_{[i]}(a)   \sum_{\substack{n_1, n_2 \equiv 1 \bmod {(1+i)^3}\\ (n_1n_2,a)=1}} \leg {(1+i)^5}{n_1n_2} \frac{d_{[i]}(n_1)d_{[i]}(n_2)}{N(n_1n_2)^{\frac{1}{2}}}\sum_{(d,2)=1}\leg { d}{n_1n_2} h(a^2d,n_1,n_2)\\
=& 2X \sum_{\substack{a \equiv 1 \bmod {(1+i)^3} \\ N(a) \leq Y }}   \frac{\mu_{[i]}(a)}{N(a)^2} \sum_{k \in
   \mathcal{O}_K}(-1)^{N(k)} \sum_{\substack{n_1, n_2 \equiv 1 \bmod {(1+i)^3}\\ (n_1n_2,a)=1}} \frac{d_{[i]}(n_1)d_{[i]}(n_2)}{N(n_1n_2)^{\frac{1}{2}}}\frac{g(k, n_1n_2)}{N(n_1n_2)}\widetilde{h}\Big(\sqrt{\frac {N(k)X}{2N(a^2n_1n_2)}}, n_1, n_2  \Big ),
\end{split}
\end{align}
  where
\begin{align*}
 \widetilde{h}(t, y,z) =  \widetilde{\Phi}(t)
V\left(\frac{ N(y)}{U_1} \right)V
\left(\frac{ N(z)}{U_2} \right).
\end{align*}

Now we write $S_1 = S_1(k=0) + S_1(k \neq 0)$, where $S_1(k=0)$ corresponds to the term with $k=0$. By applying \eqref{eq:Vdef} and \eqref{Phitilde}, we see that when $k \neq 0$,
\begin{align}
\label{wildh}
\begin{split}
 & \widetilde{h}\Big(\sqrt{\frac {N(k)X}{2N(a^2n_1n_2)}}, n_1, n_2  \Big ) \\
 =&
\frac{\pi }{(2\pi i)^3 }\int\limits_{(\varepsilon)}\int\limits_{(\varepsilon)}\int\limits_{(\varepsilon)}
 \left(\frac{N(a)^2}{ N(k)}\right)^s \mathcal{J}(s) w(u)w(v) \frac{1}{N(n_1)^{u-s} N(n_2)^{v-s}}\frac{U_1^u U_2^vX^{-s}}{uv}\ ds \ du \ dv, \\
 =&
\frac{\pi }{(2\pi i)^3 }\int\limits_{(\varepsilon)}\int\limits_{(\varepsilon)}\int\limits_{(\half+\varepsilon)}
 \left(\frac{N(a)^2}{ N(k)}\right)^s \mathcal{J}(s) w(u+s)w(v+s) \frac{1}{N(n_1)^{u} N(n_2)^{v}}\frac{U_1^{u+s} U_2^{v+s}X^{-s}}{(u+s)(v+s)}\ ds \ du \ dv,
\end{split}
\end{align}
  where
\begin{align*}
 \mathcal{J}(s)=\widehat{\Phi}(1-s)
(\frac {\pi^2}{2} )^{-s}\frac{\Gamma (s)}{\Gamma (1-s)}
\end{align*}
   and where the last expression in \eqref{wildh} follows from moving the lines of the first triple integral in \eqref{wildh} to $\Re(s) = \frac{1}{2}+\varepsilon,\Re(u)=\Re(v)=\frac{1}{2}+2\varepsilon$, and a change the variables $u'=u-s,v'=v-s$.

Substituting the last expression in  \eqref{wildh} to \eqref{equ:13.a}, we see  by using our notation for $k_1, k_2$ given in Section \ref{sect: alybehv} that
\begin{align}
\label{equ:15.3}
\begin{split}
S_1(k \neq 0)=& 2X \sum_{\substack{a \equiv 1 \bmod {(1+i)^3} \\ N(a) \leq Y }}   \frac{\mu_{[i]}(a)}{N(a)^2} \sum_{\substack{k \in
    \mathcal{O}_K \\ k \neq 0}}(-1)^{N(k)}
\frac{\pi }{(2\pi i)^3 }\int\limits_{(\varepsilon)}\int\limits_{(\varepsilon)}\int\limits_{(\half+\varepsilon)}
 \left(\frac{N(a)^2}{ N(k)}\right)^s \mathcal{J}(s) w(u+s)w(v+s)\\
& \times \frac{U_1^{u+s} U_2^{v+s}X^{-s}}{(u+s)(v+s)}
 Z(\half+u, \half+v, a, k)
 \ ds \ du \ dv \\
=& \frac{2\pi X }{(2\pi i)^3 }\sum_{\substack{a \equiv 1 \bmod {(1+i)^3} \\ N(a) \leq Y }}   \frac{\mu_{[i]}(a)}{N(a)^2} \sum_{\substack{k \in
    \mathcal{O}_K \\ k \neq 0}}\frac{(-1)^{N(k)}}{N(k)^s} \int\limits_{(\varepsilon)}\int\limits_{(\varepsilon)}\int\limits_{(\half+\varepsilon)}
 N(a)^{2s}\mathcal{J}(s) w(u+s)w(v+s) \\
& \quad\times \frac{U_1^{u+s} U_2^{v+s}X^{-s}}{(u+s)(v+s)}  L^2(1+u,\chi_{ik_1}) L^2(1+v,\chi_{ik_1})Z_2(\tfrac{1}{2}+u,\tfrac{1}{2}+v,a,k)\ ds \ du \ dv.
\end{split}
\end{align}

  We observe that if we move the lines of integrations over $u,v$ in the last expression in \eqref{equ:15.3} to the left, then we encounter poles at $u=v=0$ only when $k_1=\pm i$. For this reason, we further write $S_1(k \neq 0)=S_1(k_1=\pm i)+S_1(k_1 \neq \pm i)$, where
\begin{align}
\label{S1k1i}
\begin{split}
S_1(k_1=\pm i)
 =& \frac{2\pi X }{(2\pi i)^3 }\sum_{\substack{a \equiv 1 \bmod {(1+i)^3} \\ N(a) \leq Y }}   \frac{\mu_{[i]}(a)}{N(a)^2} \sum_{\substack{k_2 \in
   G }}\frac {(-1)^{N(k_2)}}{N(k_2)^{2s}} \int\limits_{(\varepsilon)}\int\limits_{(\varepsilon)}\int\limits_{(\half+\varepsilon)}
 N(a)^{2s}\mathcal{J}(s) w(u+s)w(v+s)    \\
& \times \frac{U_1^{u+s} U_2^{v+s}X^{-s}}{(u+s)(v+s)}  \zeta^2_K(1+u) \zeta_K^2(1+v)Z_2(\tfrac{1}{2}+u,\tfrac{1}{2}+v,a, \pm i k_2^2)\ ds \ du \ dv,
\end{split}
\end{align}
and
\begin{align}
\label{equ:15.4}
\begin{split}
 S_1(k_1 \neq \pm i)
  =& \frac{2\pi X }{(2\pi i)^3 }\sum_{\substack{a \equiv 1 \bmod {(1+i)^3} \\ N(a) \leq Y }}   \frac{\mu_{[i]}(a)}{N(a)^2} \sum_{\substack{k \in
   \mathcal{O}_K \\ k \neq 0, k_1 \neq \pm i}}\frac{(-1)^{N(k)} }{N(k)^s} \int\limits_{(\varepsilon)}\int\limits_{(\varepsilon)}\int\limits_{(\half+\varepsilon)}
 N(a)^{2s}\mathcal{J}(s) w(u+s)w(v+s)  \\
& \quad\times \frac{U_1^{u+s} U_2^{v+s}X^{-s}}{(u+s)(v+s)} L^2(1+u,\chi_{ik_1}) L^2(1+v,\chi_{ik_1})Z_2(\tfrac{1}{2}+u,\tfrac{1}{2}+v,a,k)\ ds \ du \ dv.
\end{split}
\end{align}

\subsection{Computing $S_1$: the term $S_1(k=0)$}
Note that by Lemma \ref{Gausssum} we have $g(0,n)=\varphi_{[i]}(n)$ if $n = \square$, and $0$ otherwise.   Thus we get
\begin{align}
\label{k=0}
  S_1(k=0)=& 2X \sum_{\substack{n_1, n_2 \equiv 1 \bmod {(1+i)^3}\\ n_1n_2=\square}} \frac{d_{[i]}(n_1)d_{[i]}(n_2)}{N(n_1n_2)^{\frac{1}{2}}}\frac{\varphi_{[i]}(n_1n_2)}{N(n_1n_2)} \sum_{\substack{ a \equiv 1 \bmod {(1+i)^3} \\(a,n_1n_2)=1 \\ N(a) \leq Y}} \frac {\mu_{[i]}(a)}{N(a)^2}\widetilde{h}\Big(0, n_1, n_2  \Big ).
\end{align}

    We note that
\begin{align*}
& \sum_{\substack{ a \equiv 1 \bmod {(1+i)^3} \\ (a,n_1n_2)=1 \\ N(a) \leq Y}} \frac {\mu_{[i]}(a)}{N(a)^2}
= \frac{4}{3\zeta_K(2)}\prod_{\substack{ \varpi \equiv 1 \bmod {(1+i)^3} \\ \varpi |n_1n_2}}\left(1-\frac{1}{N(\varpi)^2} \right)^{-1} +O \left(Y^{-1}\right).
\end{align*}

   Applying this to \eqref{k=0}, we see that
\begin{align*}
\begin{split}
S_1(k=0)
=\frac{8X}{3\zeta_K(2)}\sum_{\substack{n_1, n_2 \equiv 1 \bmod {(1+i)^3}\\  n_1n_2=\square}} \frac{d_{[i]}(n_1)d_{[i]}(n_2)}{N(n_1n_2)^{\frac{1}{2}}} \mathcal{P}(n_1n_2)\widetilde{h}\Big(0, n_1, n_2  \Big ) \\
 + O \Big(\frac{X}{Y} \sum_{\substack{n_1, n_2 \equiv 1 \bmod {(1+i)^3}\\  n_1n_2=\square}} \frac{d_{[i]}(n_1)d_{[i]}(n_2)}{N(n_1n_2)^{\frac{1}{2}}} \widetilde{h}\Big(0, n_1, n_2  \Big ) \Big ),
\end{split}
\end{align*}
  where we recall that $\mathcal{P}(n_1n_2)$ is defined in \eqref{Pn}.

   Due to the rapid decay of $\widetilde{h}$, we can estimate the error term above as
\begin{align*}
 \sum_{\substack{n_1, n_2 \equiv 1 \bmod {(1+i)^3}\\  n_1n_2=\square}} \frac{d_{[i]}(n_1)d_{[i]}(n_2)}{N(n_1n_2)^{\frac{1}{2}}} \widetilde{h}\Big(0, n_1, n_2  \Big ) \ll \sum_{\substack{n_1, n_2 \equiv 1 \bmod {(1+i)^3}\\  n_1n_2=\square \\ N(n_1), N(n_2) \ll X^{1+\varepsilon}}} \frac{d_{[i]}(n_1)d_{[i]}(n_2)}{N(n_1n_2)^{\frac{1}{2}}}.
\end{align*}

   We now denote $d=(n_1, n_2)$ with $d$ being primary together with a change of variables: $n_j \rightarrow dn_j, j=1,2$ to see that for the new variables $n_1, n_2$, we have $(n_1, n_2)=1$ and the condition that $d^2n_1n_2=\square$ further implies that both $n_j, j=1,2$ are squares now so that we have
\begin{align*}
& \sum_{\substack{n_1, n_2 \equiv 1 \bmod {(1+i)^3}\\  n_1n_2=\square \\ N(n_1), N(n_2) \ll X^{1+\varepsilon}}} \frac{d_{[i]}(n_1)d_{[i]}(n_2)}{N(n_1n_2)^{\frac{1}{2}}} \ll \sum_{\substack{d \equiv 1 \bmod {(1+i)^3}\\  N(d) \ll X^{1+\varepsilon}}}  \sum_{\substack{n_1, n_2 \equiv 1 \bmod {(1+i)^3}\\  n_1, n_2=\square \\ N(n_1), N(n_2) \ll X^{1+\varepsilon}/N(d)}} \frac{d_{[i]}(dn_1)d_{[i]}(dn_2)}{N(d^2n_1n_2)^{\frac{1}{2}}} \\
\ll &  \sum_{\substack{d \equiv 1 \bmod {(1+i)^3}\\  N(d) \ll X^{1+\varepsilon}}} \frac{d_{[i]}(d)^2}{N(d)} \sum_{\substack{n_1, n_2 \equiv 1 \bmod {(1+i)^3}\\  n_1, n_2=\square \\ N(n_1), N(n_2) \ll X^{1+\varepsilon}/N(d)}} \frac{d_{[i]}(n_1)d_{[i]}(n_2)}{N(n_1n_2)^{\frac{1}{2}}} \\
\ll &  \sum_{\substack{d \equiv 1 \bmod {(1+i)^3}\\  N(d) \ll X^{1+\varepsilon}}} \frac{d_{[i]}(d)^2}{N(d)} \Big ( \sum_{\substack{n \equiv 1 \bmod {(1+i)^3} \\ N(n)^2 \ll X^{1+\varepsilon}/N(d)}} \frac{d_{[i]}(n^2)}{N(n)} \Big )^2 \ll (\log X)^6\sum_{\substack{d \equiv 1 \bmod {(1+i)^3}\\  N(d) \ll X^{1+\varepsilon}}} \frac{d_{[i]}(d)^2}{N(d)} \ll (\log X)^{10}.
\end{align*}
  where the last line follows by applying estimations similar to that given in \eqref{destimation}.

  Further note that we have
\begin{align*}
  \int\limits^{\infty}_{-\infty}\int\limits^{\infty}_{-\infty}\Phi \left(N(x+yi) \right)
\dif x \dif y =\int^{2\pi}_0\int^{\infty}_0\Phi (r^2) rdrd\theta =\pi \widehat{\Phi}(1).
\end{align*}

  We now conclude from the above discussions that
\begin{align}
\label{equ:14.2}
\begin{split}
S_1(k=0)
=& \frac{8X}{3\zeta_K(2)}\sum_{\substack{n_1, n_2 \equiv 1 \bmod {(1+i)^3}\\  n_1n_2=\square}} \frac{d_{[i]}(n_1)d_{[i]}(n_2)}{N(n_1n_2)^{\frac{1}{2}}} \mathcal{P}(n_1n_2)\widetilde{h}\Big(0, n_1, n_2  \Big )
 + O \left(\frac{X}{Y}(\log X)^{10} \right) \\
=& \frac{8\pi \widehat{\Phi}(1)X}{3\zeta_K(2)}\frac {1}{(2\pi
   i)^2} \int\limits\limits_{(2)}\int\limits\limits_{(2)}\frac {w(u)w(v)}{uv} U_1^u U_2^v Z(\half+u, \half+v)dudv
 + O \left(\frac{X}{Y}(\log X)^{10} \right),
\end{split}
\end{align}
  where $Z(\half+u, \half+v)$ is defined as in \eqref{zalphabeta}.

 It follows from  \eqref{equ:14.2} and  Lemma \ref{lem: zalphabeta} that
\begin{align}
\label{equ:14.2+}
S_1(k=0)
=\frac{8\pi \widehat{\Phi}(1)X}{3\zeta_K(2)}\frac {1}{(2\pi
   i)^2} \int\limits\limits_{(2)}\int\limits\limits_{(2)}\frac{U_1^u U_2^v}{uv(2u)^3(2v)^3(u+v)^4} \mathcal{E}(u,v)  \ du\ dv
+O \left(XY^{-1}(\log X)^{10} \right),
\end{align}
where
\[
\mathcal{E}(u,v) = w(u)w(v)  \zeta_K^3(1+2u)(2u)^3\zeta_K^3(1+2v)(2v)^3  \zeta_K^4(1+u+v)(u+v)^4  Z_1(\tfrac{1}{2}+u,\tfrac{1}{2}+v).
\]
Applying  Lemma \ref{lem: zalphabeta} again, we see that $\mathcal{E}$ is analytic for $\Re(u),\Re(v)> -\frac{1}{4} + \varepsilon$.

 We first move the lines of the integrations in \eqref{equ:14.2+} to $\Re(u) = \Re(v) = \frac{1}{10}$ by noting that we encounter no poles. We then move the line of the integration over $v$ to $\Re(v) = -\frac{1}{5}$ to see that we encounter two poles of order $4$ at $v=0$ and $v=-u$ in the process. It follows that
\begin{align}
 \label{equ:14.3}
\begin{split}
&\frac{1}{(2\pi i)^2}\int\limits_{(\frac{1}{10})}\int\limits_{(\frac{1}{10})}\frac{U_1^u U_2^v}{uv(2u)^3(2v)^3(u+v)^4} \mathcal{E}(u,v) \ du\ dv  \\
&= \frac{1}{2\pi i}\int\limits_{(\frac{1}{10})} \left( \underset{v=0}{\operatorname{Res}} +\underset{v=-u}{\operatorname{Res}}\right)
 \left ( \frac{U_1^u U_2^v}{uv(2u)^3(2v)^3(u+v)^4}\mathcal{E}(u,v)  \right )du + O \left(U_1^{\frac{1}{10}} U_2^{-\frac{1}{5}} \right).
\end{split}
\end{align}

  We treat the contribution from the residue at $v=0$ in \eqref{equ:14.3} to see that
\begin{align*}
I_1(u)&=\underset{v=0}{\operatorname{Res}} \left(\frac{U_1^u U_2^v}{uv(2u)^3(2v)^3(u+v)^4}\mathcal{E}(u,v)  \right) \\
&= \frac{U_1^u}{2^6\cdot3!u^{11}}
\Big (
\mathcal{E}(u,0) (u^3 (\log U_2)^3 -12 u^2 (\log U_2)^2 + 60 u \log U_2 -120 ) \\
& \quad +
\mathcal{E}^{(0,1)}(u,0)(3u^3 (\log U_2)^2 -24 u^2  \log U_2 + 60 u )+\mathcal{E}^{(0,2)}(u,0)(3u^3 \log U_2 - 12u^2) + \mathcal{E}^{(0,3)}(u,0) u^3
\Big ),
\end{align*}
  where $ \mathcal{E}^{(i,j)}(u, v) = \frac{\partial^{i+j}\mathcal{E}}{\partial u^i \partial v^j} (u,v)$.

 It follows by moving the line of the integration over $u$ from $\Re(u)= \frac{1}{10}$ to $\Re(u)= -\frac{1}{10}$ that we have
\begin{align*}
\begin{split}
\frac{1}{2\pi i}\int\limits_{(\frac{1}{10})} I_1(u) du
&=
 \underset{u=0}{\operatorname{Res}} \ I_1(u)
  +O(U_1^{-\frac{1}{10}}\log^3 X) \\
&=  \frac{5! \mathcal{E}(0,0)}{2^6\cdot 3!\cdot 10!} \left(-(\log U_1)^{10} + 5 (\log U_1)^9 \log U_2 - 9(\log U_1)^8 (\log U_2)^2
+ 6 (\log U_1)^7 (\log U_2)^3
 \right) \\
 &\quad +O\left((\log X)^9+ U_1^{-\frac{1}{10}}(\log X)^3 \right).
\end{split}
\end{align*}

  Similarly, we have that
\begin{align*}
\frac{1}{2\pi i}\int\limits_{(\frac{1}{10})} \underset{v=-u}{\operatorname{Res}}
 \left( \frac{U_1^u U_2^v}{uv(2u)^3(2v)^3(u+v)^4}\mathcal{E}(u,v)  \right)du \ll U_1^{\frac{1}{10}} U_2^{-\frac{1}{10}} (\log X)^3.
\end{align*}

Combining \eqref{equ:14.2+}-\eqref{equ:14.3}, together with the observation that $\displaystyle \lim_{s \rightarrow 0}\zeta_K(1+s)s=\pi/4$ and $Z_1(\frac{1}{2},\frac{1}{2}) = 4 a_4$, we obtain that
\begin{align}
\label{S1k0est}
\begin{split}
& S_1(k=0)  \\
=&\frac{8\pi X}{3\zeta_K(2)} \widehat{\Phi}(1) \cdot \leg {\pi}{4}^{10} \frac{4 \cdot 5! a_{4 }}{2^6 \cdot 3!\cdot 10!} \left(-(\log U_1)^{10} + 5 (\log U_1)^9 \log U_2 - 9(\log U_1)^8 (\log U_2)^2
+ 6 (\log U_1)^7 (\log U_2)^3
 \right) \\
& + O\left( X(\log X)^9+ XY^{-1} (\log X)^{10} \right).
\end{split}
\end{align}

\subsection{Computing $S_1$: the term  $S_1(k_1=\pm i)$}
\label{sec:k=square}

In this section, we evaluate $S_1(k_1=\pm i)$ given in \eqref{S1k1i}. Applying Lemma \ref{lem:15.3}, we see that
\begin{align*}
S_1(k_1=\pm i)
  =& \frac{2\pi X }{(2\pi i)^3 }\sum_{\substack{a \equiv 1 \bmod {(1+i)^3} \\ N(a) \leq Y }}   \frac{\mu_{[i]}(a)}{N(a)^2} \sum_{\substack{k_2 \in
   \mz[i] }}(-1)^{N(k_2)} \int\limits_{(\varepsilon)}\int\limits_{(\varepsilon)}\int\limits_{(\half+\varepsilon)}
 N(a)^{2s}\mathcal{J}(s) w(u+s)w(v+s) \frac{U_1^{u+s} U_2^{v+s} X^{-s}}{(u+s)(v+s)}   \\
  & \times \zeta^2_K(1+u)\zeta^2_K(1+v) (2^{1-2s}-1)\zeta_K(2s) \zeta_K^2(1+2u +2s) \zeta_K^2(1+2v +2s) Z_3(\tfrac{1}{2}+u,\tfrac{1}{2}+v, s,a)\ ds \ du \ dv.
\end{align*}

  As $Z_3(\tfrac{1}{2}+u,\tfrac{1}{2}+v, s)$ is analytic in the region $\Re(u),\Re(v) \geq \varepsilon, \Re(s) \geq 2\varepsilon$ by  (1) of Lemma \ref{lem:15.3} and $(2^{1-2s} -1) \zeta_K(2s)$ is analytic at $s=\frac{1}{2}$, we move the lines of the integration above to $\Re(u) = \Re(v) = 1,\Re(s) =\frac{1}{10}$ without encountering any poles to see that
\begin{align}
\label{equ:15.9-}
\begin{split}
S_1(k_1=\pm i)
  =& \frac{2\pi X }{(2\pi i)^3 }\sum_{\substack{a \equiv 1 \bmod {(1+i)^3} \\ N(a) \leq Y }}   \frac{\mu_{[i]}(a)}{N(a)^2} \int\limits_{(1)}\int\limits_{(1)}\int\limits_{(\frac{1}{10})}
 N(a)^{2s}\mathcal{J}(s) w(u+s)w(v+s) \frac{U_1^{u+s} U_2^{v+s} X^{-s}}{(u+s)(v+s)}   \\
 & \times \zeta^2_K(1+u)\zeta^2_K(1+v) (2^{1-2s}-1)\zeta_K(2s) \zeta_K^2(1+2u +2s) \zeta_K^2(1+2v +2s) Z_3(\tfrac{1}{2}+u,\tfrac{1}{2}+v, s,a)\ ds \ du \ dv.
\end{split}
\end{align}

 We extend the sum over $a$ in \eqref{equ:15.9-} to include all primary elements in $\mathcal{O}_K$, introducing an error term
\begin{align}
\label{alarge}
\begin{split}
&\frac{2\pi X }{(2\pi i)^3 }\sum_{\substack{a \equiv 1 \bmod {(1+i)^3} \\ N(a) > Y }}   \frac{\mu_{[i]}(a)}{N(a)^2} \int\limits_{(1)}\int\limits_{(1)}\int\limits_{(\frac{1}{10})}
 N(a)^{2s}\mathcal{J}(s) w(u+s)w(v+s) \frac{U_1^{u+s} U_2^{v+s} X^{-s}}{(u+s)(v+s)}   \\
 \times & \zeta^2_K(1+u)\zeta^2_K(1+v) (2^{1-2s}-1)\zeta_K(2s) \zeta_K^2(1+2u +2s) \zeta_K^2(1+2v +2s) Z_3(\tfrac{1}{2}+u,\tfrac{1}{2}+v, s,a)\ ds \ du \ dv.
 \end{split}
\end{align}

  To facilitate our estimation of the triple integral in the above expression and other similar integrals in what follows, we gather here a few bounds on $\zeta_K(s)$ that hold uniformly in specified regions. On write $s=\sigma+it$,  we have
\begin{align}
\label{zetaest}
\begin{split}
  \zeta_K(s) \ll & \big(1+(|t|+4)^{1-\sigma} \big )\min \big(\frac 1{|\sigma-1|}, \log (|t|+4) \big ), \quad \sigma>1, \\
  \zeta_K(s) \ll & \left( 1+|s|^2 \right)^{1-\sigma/2+\varepsilon}, \quad 0 \leq \sigma \leq 1, \\
  \frac 1{\zeta_K(s)} \ll & \log (|t|+4), \quad 1 \leq \sigma \leq 2.
\end{split}
\end{align}
  The first and third estimation above can be established similar to the proofs of \cite[Corollary 1.17]{MVa1} and \cite[Lemma 6.7]{MVa1}, respectively. The second estimation above is the convexity bound for $\zeta_K(s)$ (see \cite[Exercise 3, p. 100]{iwakow}).

 Also, by applying \eqref{Gausscircle}, we see that for the $j$-th derivative of $\zeta_K(s)$ with $j \geq 1$, we have for $\Re(s)>1$,
\begin{align*}
  \zeta^{(j)}_K(s)=& \sum_{\substack{ (n) \\ N(n) > 1}}(-\log N(n))^j N(n)^{-s}=\int^{\infty}_1 (-\log u)^j u^{-s}d(\sum_{\substack{ (n) \\ N(n) \leq u}}1) \\
=& \int^{\infty}_1 (-\log u)^j u^{-s}d(\frac {\pi}{4} u+O(u^{\theta}))  \\
=& \frac {\pi}{4}\frac {j!}{(1-s)^j}+O\big( \int^{\infty}_1 (s+\frac {j}{\log u})(-\log u)^j u^{-s-1+\theta}du \big).
\end{align*}
   We deduce readily from the above that for $j \geq 1$ and $\Re(s)>0.4$, we have
\begin{align}
\label{zetaderbound}
 \zeta^{(j)}_K(s) \ll 1+|s|.
\end{align}

  We further note that integrating by parts implies that for $\Re(s)<1$ and any integer $\nu \geq 1$,
\begin{align}
\label{equ:bd for fei}
\widehat{\Phi}(1-s) \ll _{\nu} \frac{3^{|\Re(s)|}}{|s||s-1|^{\nu-1}} \Phi_{(\nu)},
\end{align}
  where
\[
\Phi_{(\nu)} = \underset{0 \leq j \leq \nu}{\operatorname{ max }} \int_{\frac{1}{2}} ^{\frac{5}{2}} | \Phi^{(j)} (t)| dt.
\]

We move the lines of the integrations in \eqref{alarge} over $u,v$ to $\Re(u)=\Re(v)=\frac{1}{\log X}, \Re(s) = \frac{2}{\log X}$ without encountering any poles. Then by (2) of Lemma \ref{lem:15.3} and the estimations given in \eqref{zetaest} and \eqref{equ:bd for fei}, we see that on the new lines of integrations, the expression in \eqref{alarge} is
\begin{align*}
&\ll X (\log X)^{20} \sum_{\substack{a > Y \\ (a,2)=1}} \frac{1}{a^{2-\frac{4}{\log X}}} \int\limits_{(\frac{1}{\log X})}\int\limits_{(\frac{1}{\log X})}\int\limits_{(\frac{2}{\log X})}\
 (1+|2s|)^{1+\varepsilon} |\mathcal{J}(s)|  \left|\Gamma(u+s+\tfrac{1}{2})^2  \Gamma(v+s+\tfrac{1}{2})^2 \right|\ |ds| \ |du| \ |dv| \\
&\ll X (\log X)^{20} Y^{-1} \int_{(\frac{2}{\log X})} (1+|2s|)^{1+\varepsilon}|\frac {\Gamma(s)}{\Gamma (1-s)}| |\widehat{\Phi}(1-s)| |ds| \ll X  Y^{-1} (\log X)^{22} \Phi_{(4)},
\end{align*}
 where the last estimation above follows from \eqref{equ:bd for fei} with $\nu=4$ and the bound $\frac {\Gamma(s)}{\Gamma (1-s)} \ll |s|^{2\Re(s)-1}$.

We conclude from the above discussions and \eqref{equ:15.9-} that
\begin{align}
\label{equ:15.9}
\begin{split}
& S_1(k_1=\pm i) \\
  =& \frac{2\pi X }{(2\pi i)^3 }\sum_{\substack{a \equiv 1 \bmod {(1+i)^3}  }}   \frac{\mu_{[i]}(a)}{N(a)^2} \int\limits_{(\varepsilon)}\int\limits_{(\varepsilon)}\int\limits_{(\half+\varepsilon)}
 N(a)^{2s}\mathcal{J}(s) w(u+s)w(v+s) \frac{U_1^{u+s} U_2^{v+s} X^{-s}}{(u+s)(v+s)}   \\
 & \times \zeta^2_K(1+u)\zeta^2_K(1+v) (2^{1-2s}-1)\zeta_K(2s) \zeta_K^2(1+2u +2s) \zeta_K^2(1+2v +2s) Z_3(\tfrac{1}{2}+u,\tfrac{1}{2}+v, s,a)\ ds \ du \ dv \\
&+ O \left(  X  Y^{-1} (\log X)^{22}  \Phi_{(4)} \right) .
\end{split}
\end{align}

 We further apply  Lemma \ref{lem:15.4} to deduce from \eqref{equ:15.9} that
\begin{align}
\label{equ:15.10}
\begin{split}
 & S_1(k_1=\pm i) \\
  =& \frac{2\pi X }{(2\pi i)^3 } \int\limits_{(\varepsilon)}\int\limits_{(\varepsilon)}\int\limits_{(\half+\varepsilon)}
\mathcal{J}(s)(2^{1-2s}-1)\zeta_K(2s) w(u+s)w(v+s) \frac{U_1^{u+s} U_2^{v+s} X^{-s}}{(u+s)(v+s)}   \\
&\times  \zeta_K^2(1+u)\zeta_K^2(1+v)
 \frac{ \zeta_K^3(1+2u+2s)\zeta_K^3(1+2v+2s) \zeta_K^4(1+u+v+2s)}{\zeta_K^2(1+u+2s)\zeta_K^2(1+v+2s)} \\
& \times Z_4(\tfrac{1}{2}+u,\tfrac{1}{2}+v,s) \ ds \ du \ dv + O \left(X  Y^{-1}(\log X)^{22} \Phi_{(4)} \right).
\end{split}
\end{align}

  We now move the lines of the integrations in \eqref{equ:15.10} to $\Re(u) = \Re(v) = \Re(s) = \frac{1}{100}$ without encountering any poles, as $Z_4(\tfrac{1}{2}+u,\tfrac{1}{2}+v,s)$ is analytic and uniformly bounded in the region  $\Re(u),\Re(v)\geq -\frac{1}{8}, -\frac{1}{16}  \leq \Re(s) \leq \frac{1}{8} $ by Lemma \ref{lem:15.4}. Next, we move the line of the integration over $v$ to $\Re(v)= -\frac{1}{50}  + \frac{1}{\log X}$ to encounter a pole of order  $2$ at $v=0$ and a pole of order $4$ at $v=-s$ to see that triple integral in \eqref{equ:15.10} is
\begin{align}
\label{equ:15.12}
  \frac{1}{(2\pi i)^2} \int\limits_{(\frac{1}{100})} \int\limits_{(\frac{1}{100})} ( I_2  (u,s) + I_3 (u,s) )
 \ du \ ds + O \left(U_1^{\frac{1}{50}} U_2^{-\frac{1}{100}}X^{-\frac{1}{100}}   \Phi_{(4)} \right),
\end{align}
where the error term above follows from using estimations given in \eqref{zetaest} and \eqref{equ:bd for fei} and where we denote $I_2(u,s),I_3(u,s)$ for the residues of the integrand in \eqref{equ:15.10} at $v=0$ and $v=-s$, respectively.

 We evaluate the double integral of $I_2 (u,s)$ in \eqref{equ:15.12} by writing the integrand in  \eqref{equ:15.10} as
\[
\frac{U_1^{u+s} U_2^{v+s}X^{-s}}{(u+s)(v+s)}    \frac{1}{u^2v^2}   \frac{(u+2s)^2 (v+2s)^2}{s (2u+2s)^3 (2v+2s)^3 (u+v+2s)^4}\mathcal{F}(u,v,s),
\]
  where
\begin{align*}
\mathcal{F}(u,v,s)
  =&
\mathcal{J}(s)(2^{1-2s}-1)\zeta_K(2s)(s) w(u+s)w(v+s)   \\
&\quad \times  (\zeta_K(1+u)u)^2(\zeta_K(1+v)v)^2 \\
& \times \frac{ (\zeta_K(1+2u+2s)(2u+2s))^3(\zeta_K(1+2v+2s)(2v+2s))^3 (\zeta_K(1+u+v+2s)(u+v+2s))^4}{(\zeta_K(1+u+2s)(u+2s))^2(\zeta_K(1+v+2s)(v+2s))^2} \nonumber\\
&\quad \times Z_4(\tfrac{1}{2}+u,\tfrac{1}{2}+v,s).
\end{align*}

  It is easy to see that $\mathcal{F}(u,v,s)$ is analytic for $\Re(u+2s),\Re(v+2s)>0$ by Lemma \ref{lem:15.4} and that
\begin{align*}
I_2(u,s)
=&  \frac{U_1^{u+s} U_2^{s} X^{-s}}{ 16(u+2s)^3 (u+s)^4 s^4 u^2}\left(
\mathcal{F} (u,0,s) ( s(u+2s) \log U_2  -10s -3u)
+
\mathcal{F}^{(0,1,0)} (u,0,s)  s(u+2s)
\right ).
\end{align*}
 It follows from this that by moving the line of the double integral in \eqref{equ:15.12} involving $I_2(u,s)$ from $\Re(u)= \frac{1}{100} $ to $\Re(u)= -\frac{1}{100} + \frac{1}{\log X}$ and applying \eqref{zetaest}, \eqref{zetaderbound}, we have
\begin{align}
\label{equ:15.14}
\frac{1}{(2\pi i)^2} \int\limits_{(\frac{1}{100})} \int\limits_{(\frac{1}{100})}
I_2(u,s)  \ du\ ds
= \frac{1}{2\pi i} \int\limits_{(\frac{1}{100})}
\underset{u=0}{\operatorname{Res}}
 \left( I_2(u,s) \right) ds
+ O \left(U_2^{\frac{1}{100}}X^{-\frac{1}{100}} (\log X)^{5} \right).
\end{align}

Note that
\begin{align*}
\underset{u=0}{\operatorname{Res}}\ I_2(u,s)
=& \frac{U_1^{s} U_2^s X^{-s}}{64 s^{11} }
 \Big(
\mathcal{F}(0,0,s)(s^2\log U_1 \log U_2
-5  s\log U_1
-5 s \log U_2
+26
  )\\
& +\mathcal{F}^{(1,0,0)}(0,0,s)(s^2 \log U_2
-5   s)
+\mathcal{F}^{(0,1,0)}(0,0,s)(s^2\log U_1
-5   s)
+ \mathcal{F}^{(1,1,0)}(0,0,s) s^2
\Big).
\end{align*}

 As one checks that the expression in the parenthesis above is analytic for $-\frac{1}{16} \leq \Re(s) \leq \frac{1}{8}$,
we can move the line of the integral in \eqref{equ:15.14} involving $\displaystyle \underset{u=0}{\text{Res }} I_2(u,s)$ to $\Re(s)=-\frac{1}{100}$. In this process, we encounter a pole at $s=0$ so that we have
\begin{align}
\label{equ:15.15}
\begin{split}
\frac{1}{2\pi i} \int\limits_{(\frac{1}{100})}
\underset{u=0}{\operatorname{Res}}
 \left( I_2(u,s) \right) ds
=&
\frac{\mathcal{F}(0,0,0)}{64} \sum_{\substack{j_1 + j_2 +j_3 +j_4 = 10 \\ j_1,j_2,j_3,j_4 \geq 0}} \frac{(-1)^{j_3} B(j_4)}{j_1 ! j_2 ! j_3 ! j_4 !}
 (\log^{j_1} U_1)  (\log^{j_2} U_2) (\log^{j_3} X)\\
 & + O \left( U_1^{-\frac{1}{100}} U_2^{-\frac{1}{100}} X^{\frac{1}{100}} + (\log X)^9 \right),
\end{split}
\end{align}
where
 \begin{align*}
 B(j)=\left\{
 \begin{array}
  [c]{ll}
     26                    & \text{if } j=0,\\
  -5 (\log U_1 + \log U_2) & \text{if } j=1,\\
  2 \log U_1 \log U_2      & \text{if } j= 2,\\
  0                        & \text{if } j\geq 3.
 \end{array}
 \right.
\end{align*}

   To evaluate $\mathcal{F}(0,0,0)$, we use the fact that $s\Gamma(s)=1$ when $s=0$ and the functional equation \eqref{fneqn} for $\zeta_K(s)$:
\begin{align*}
\pi^{-s}\Gamma(s)\zeta_K(s)=\pi^{-(1-s)}\Gamma(1-s)\zeta_K(1-s)
\end{align*}
  to obtain that $\zeta_K(0)=-\frac {1}{4}$. On the other hand, a direct calculation shows that $Z_4(\tfrac{1}{2},\tfrac{1}{2},0) =\frac {16}{3\zeta_K(2)}a_4$. 
  We then deduce that
\begin{align*}
\mathcal{F}(0,0,0)
  =\widehat{\Phi}(1)\zeta_K(0)(\frac {\pi}{4})^{10} Z_4(\tfrac{1}{2},\tfrac{1}{2},0)= - \frac{4 \widehat{\Phi}(1) a_4}{3\zeta_K(2)}.
\end{align*}

 Similarly, by moving the line of the integration over $u$ from  $\Re(u)=\frac{1}{100}$ to $\Re(u)=\frac{1}{\log X}$ in the double integral of $I_3 (u,s)$ in \eqref{equ:15.12} and applying estimations given in \eqref{zetaest}-\eqref{equ:bd for fei}, we see that
\begin{align}
\label{equ:15.13}
 \frac{1}{(2\pi i)^2} \int\limits_{(\frac{1}{100})} \int\limits_{(\frac{1}{100})}  I_3 (u,s)
 \ du \ ds
\ll  U_1^{\frac{1}{100}}  X^{-\frac{1}{100}} ( \log   X)^5 \Phi_{(5)}.
\end{align}

 We now summarize our result on $S_1(k_1=\pm i)$ in the following lemma, by combining the estimations from \eqref{equ:15.10}-\eqref{equ:15.13}.
\begin{lemma}
\label{lem:k=square}
We have
\begin{align*}
S_1(k_1=\pm i)
  = &- \frac{8\pi a_4 \widehat{\Phi}(1) X}{3\cdot 64 \zeta_K(2)} (\frac {\pi}{4})^{10}   \sum_{j_1 + j_2 +j_3 +j_4 = 10} \frac{(-1)^{j_3} B(j_4)}{j_1 ! j_2 ! j_3 ! j_4 !}
 (\log^{j_1} U_1)  (\log^{j_2} U_2) (\log^{j_3} X) \\
& + X
 \cdot O \Big(
   (\log X)^{9}+U_1^{-\frac{1}{100}} U_2^{-\frac{1}{100}} X^{\frac{1}{100}}
  +U_2^{\frac{1}{100}}X^{-\frac{1}{100}} (\log X)^{5}  \Big ) \\
&
 + X\cdot O \Big(  U_1^{\frac{1}{50}} U_2^{-\frac{1}{100}} X^{-\frac{1}{100}}  \Phi_{(4)}
 +  Y^{-1} (\log X)^{22}  \Phi_{(4)} +U_1^{\frac{1}{100}}X^{-\frac{1}{100}} ( \log   X)^5 \Phi_{(5)}
   \Big ) .
\end{align*}
\end{lemma}

\subsection{Computing $S_1$: the term $S_1(k \neq \pm i)$}
\label{sec:error}
  In this section, we estimate $S_1(k \neq \pm i)$.  We first deduce from \eqref{equ:15.4} that
\begin{align}
\label{S_1k1gen}
\begin{split}
 S_1(k_1 \neq \pm i) \ll & X \sum_{\substack{a \equiv 1 \bmod {(1+i)^3} \\ N(a) \leq Y }}   \frac{1}{N(a)^2} \sum_{\substack{k \in
    \mathcal{O}_K \\ k \neq 0, k_1 \neq \pm i}}\Bigg |  \int\limits_{(\varepsilon)}\int\limits_{(\varepsilon)}\int\limits_{(\half+\varepsilon)}
 N(a)^{2s}\mathcal{J}(s) w(u+s)w(v+s)  \\
& \quad\times \frac{U_1^{u+s} U_2^{v+s}X^{-s}}{(u+s)(v+s)}\frac{1}{N(k)^s}  L^2(1+u,\chi_{ik_1}) L^2(1+v,\chi_{ik_1})Z_2(\tfrac{1}{2}+u,\tfrac{1}{2}+v,a,k) \ ds \ du \ dv \Bigg |.
\end{split}
\end{align}
 Let $Z$ be a parameter to be chosen later. We denote $S_{1,1}(k_1 \neq \pm i)$ for the right-hand side expression above truncated to $N(k_1) \leq Z$ and $S_{1,2}(k_1 \neq \pm i)$ for the right-hand side expression above over $N(k_1) > Z$.  For $S_{1,1}(k_1 \neq \pm i)$, we shift the the lines of the integrations to
$\Re(u)=\Re(v) = -\frac{1}{2}+ \frac{1}{\log X}, \Re(s)= \frac{3}{4}$. For $S_{1,2}(k_1 \neq \pm i)$, we shift the the lines of the integrations to $\Re(u)=\Re(v) = -\frac{1}{2}+ \frac{1}{\log X}, \Re(s)= \frac{5}{4}$. Thus we obtain via \eqref{equ:15.111111} that
\begin{align}
\label{equ:16.2}
\begin{split}
S_{1,1}(k_1 \neq \pm i) \ll &  X^{\frac{1}{4}} U_1^{\frac{1}{4}} U_2^{\frac{1}{4}} (\log X)^{10}\sum_{\substack{N(a) \leq Y \\ (a,2)=1}} \frac{d^{4}_{[i]}(a)}{\sqrt{N(a)}}\int\limits_{(-\frac{1}{2}+\frac{1}{\log X})}\int\limits_{(-\frac{1}{2}+\frac{1}{\log X})}\int\limits_{(\frac{3}{4})} |\mathcal{J}(s)w(u+s)w(v+s)| \\
& \times \sumflat_{\substack{ k_1 \neq \pm i \\ N(k_1) \leq  Z}}\frac{d^{12}_{[i]}(k_1)}{N(k_1)^{\frac{3}{4}}}  |L(1+u,\chi_{ik_1}) |^4 \ |ds| \ |du| \ |dv|,
\end{split}
\end{align}
  where $\sumflat$ denotes the sum over all square-free elements in $\mathcal{O}_K$.

  By the Cauchy-Schwarz inequality, we have that
\begin{align}
\label{equ:16.b-}
\begin{split}
& \sumflat_{\substack{ k_1 \neq \pm i \\ N(k_1) \leq  Z}}\frac{d^{12}_{[i]}(k_1)}{N(k_1)^{\frac{3}{4}}}  |L(1+u,\chi_{ik_1}) |^4
\ll \Big(\sumflat_{\substack{ k_1 \neq \pm i \\ N(k_1) \leq  Z}}\frac{d^{24}_{[i]}(k_1)}{N(k_1)}\Big)^2 \Big(\sumflat_{\substack{ k_1 \neq \pm i \\ N(k_1) \leq  Z}}\frac{1}{N(k_1)^{\frac{1}{2}}}  |L(1+u,\chi_{ik_1}) |^8 \Big )^{1/2} \\
\ll & (\log Z)^{2^{23}}\Big(\sumflat_{\substack{ k_1 \neq \pm i \\ N(k_1) \leq  Z}}\frac{1}{N(k_1)^{\frac{1}{2}}}  |L(1+u,\chi_{ik_1}) |^8 \Big )^{1/2}.
\end{split}
\end{align}
   Similar to our proof of Theorem \ref{upperbound} (and hence Corollary \ref{cor: upperbound}), we can show that under GRH, for $|\Im(u)|\leq T$,
\begin{align*}
 \sumflat_{\substack{ k_1 \neq \pm i \\ N(k_1) \leq  Z}}|L(1+u,\chi_{ik_1}) |^8  \ll Z  (\log Z)^{37}.
\end{align*}
  Applying the above estimation in \eqref{equ:16.b-}, we deduce that for $|\Im(u)| \leq  Z$,
\begin{align}
\label{equ:16.3b}
\sumflat_{\substack{ k_1 \neq \pm i \\ N(k_1) \leq  Z}} \frac{d^{12}_{[i]}(k_1)}{N(k_1)^{\frac{3}{4}}} |L(1+u,\chi_{ik_1}) |^4 \ll Z^{\frac{1}{4}} (\log Z)^{2^{24}}.
\end{align}

 Unconditionally, we have similar to Lemma \ref{lem:2.3} that
\begin{align*}
\sumflat_{\substack{ k_1 \neq \pm i \\ N(k_1) \leq  Z}} |L(1+u,\chi_{ik_1}) |^4
\ll Z^{1+\varepsilon} (1+|\Im(u)|^2)^{1+\varepsilon}.
\end{align*}
  It follows from this and partial summation that we have
\begin{align}
\label{equ:16.3a}
& \sumflat_{\substack{ k_1 \neq \pm i \\ N(k_1) \leq  Z}} \frac{d^{12}_{[i]}(k_1)}{N(k_1)^{\frac{3}{4}}}  |L(1+u,\chi_{ik_1}) |^4
\ll \sumflat_{\substack{ k_1 \neq \pm i \\ N(k_1) \leq  Z}} \frac{1}{N(k_1)^{\frac{3}{4}-\varepsilon}}  |L(1+u,\chi_{ik_1}) |^4
\ll Z^{\frac{1}{4}+\varepsilon}
(1+|\Im(u)|^2)^{1+\varepsilon}.
\end{align}

  Applying the estimation given \eqref{equ:16.3b} to \eqref{equ:16.2} for $|\Im(u)| \leq Z$ and using \eqref{equ:16.3a} otherwise, we obtain by noting the exponential decay of $w$ that
\begin{align}
\label{S11}
S_{1,1}(k_1 \neq \pm i) \ll &  X^{\frac{1}{4}} U_1^{\frac{1}{4}} U_2^{\frac{1}{4}} (\log X)^{10}Y(\log Y)^{2^4}Z^{\frac{1}{4}} (\log Z)^{2^{24}}\Phi_{(4)}.
\end{align}

  Similarly, we have
\begin{align}
\label{S12}
S_{1,2}(k_1 \neq \pm i)  \ll  X^{-\frac{1}{4}} U_1^{\frac{3}{4}} U_2^{\frac{3}{4}}Y^{3/2}(\log Y)^{2^4} Z^{-\frac{1}{4}}(\log Z)^{2^{24}}\Phi_{(4)}.
\end{align}

 We further observe that if instead we apply \eqref{equ:16.3a} to \eqref{equ:16.2} for all $\Im(u)$, then we obtain corresponding estimations for $S_{1,1}(k_1 \neq \pm i), S_{1,2}(k_1 \neq \pm i)$  by replacing $(\log Z)^{2^{24}}$ with $Z^{\varepsilon}$ in both \eqref{S11} and \eqref{S12}.  On setting $Z= U_1 U_2 Y^2 X^{-1}$ and keeping in mind that our choices of $Y$ and $Z$ will be at most powers of $X$ (see Section \ref{sec: pf1stmainthm}), we immediately derive from \eqref{S_1k1gen}, \eqref{S11} and \eqref{S12} the following result.
\begin{lemma}\label{lem:error}
Unconditionally, we have
\[
S_1(k \neq \pm i) \ll U_1^{\frac{1}{2}} U_2^{\frac{1}{2}}Y X^{\varepsilon}  \Phi_{(4)}.
\]
Under GRH, we have
\[
S_1(k \neq \pm i) \ll U_1^{\frac{1}{2}} U_2^{\frac{1}{2}} Y (\log X)^{2^{25}} \Phi_{(4)}.
\]
\end{lemma}

\subsection{Proof of Theorem \ref{main-thm}}
\label{sec: pf1stmainthm}

   We now complete the proof for Theorem \ref{main-thm} in this section. 
By setting $U_1=U_2= U=\frac{X}{(\log X)^{2^{50}}}$ and $Y=X^{\frac{1}{2}}U_1^{-\frac{1}{4}}U_2^{-\frac{1}{4}}$, we deduce from \eqref{equ:13.1+}, \eqref{S1k0est}, Lemma \ref{lem:13.1}, Lemma \ref{lem:k=square} and Lemma \ref{lem:error} that under GRH, 
\begin{align}
\label{equ:S(U,U)}
S(U_1,U_2) &= \sumstar_{(d,2)=1} \left|A_{U}(d)\right|^2 \Phi\left(\tfrac{N(d)}{X}\right)=  \frac{\pi a_{4} \widehat{\Phi}(1)}{2^7 \cdot 3^4 \cdot 5^2 \cdot 7 \cdot \zeta_K(2)} (\frac {\pi}{4})^{10} X (\log X)^{10}+ O \left(X (\log X)^{9+\varepsilon}  +  X (\log X)^{-20} \Phi_{(5)} \right).
\end{align}
  Here we note that, similar to \cite[Lemma 2.1]{sound1}, we can show that the function $V$ appearing in the definition of $A_U(d)$ given in \eqref{equ:13.1} is a real-valued function, so that we have $A^2_{U}(d))=|A_{U}(d)|^2$.

We define $B_U(d) = L(\tfrac{1}{2},\chi_{(1+i)^5d})^2 - A_U(d)$ to see that
\begin{align*}
B_U(d)
=\frac{1}{\pi i} \int\limits_{(c)} w(s) L(\tfrac{1}{2}+s,\chi_{(1+i)^5d})^2  \frac{N(d)^s-U^s}{s} ds.
\end{align*}
  We move the line of the integral in the above expression for $B_U(d)$ to $\Re(s)=0$ by realizing that $\frac{N(d)^s-U^s}{s}$ is entire there. Further applying the bound that $| \frac{N(d)^{it}-U^{it}}{it}| \ll  \log ( \frac{N(d)}{U} )$ for $t\in \mathbb{R}$, we see  that
\[
B_U(d) \ll \log \left( \frac{N(d)}{U} \right) \int_{-\infty}^{\infty} |w(it)|  |L(\tfrac{1}{2}+it,\chi_{(1+i)^5d})|^2dt.
\]
 It follows that 
\begin{align*}
\sumstar_{(d,2)=1}|B_U(d)|^2 \Phi\left(\tfrac{N(d)}{X}\right) \ll  \left( \log \frac{X}{U} \right)^2 \int_{-\infty}^{+\infty}\int_{-\infty}^{+\infty}
 |w(it_1)||w(it_2)|\sideset{}{^*}\sum_{(d,2)=1}|L(\tfrac{1}{2}+it_1,\chi_{8d})|^2|L(\tfrac{1}{2}+it_2,\chi_{8d})|^2 \Phi(\tfrac{N(d)}{X})dt_1dt_2.
\end{align*}
 We estimation the sum over $d$ above using Corollary \ref{4momentupperbound} when $|t_1|,|t_2| \leq X$ and Lemma \ref{lem:2.3} otherwise.  We then deduce  by the exponential decay of $w(it)$ in $t$ that
\begin{align}
\label{equ:B_U}
 \sumstar_{(d,2)=1}|B_U(d)|^2 \Phi\left(\tfrac{N(d)}{X}\right) \ll X (\log X)^{9+\varepsilon} .
\end{align}

  Combining \eqref{equ:S(U,U)} and  \eqref{equ:B_U}, we see that
\begin{align}
\label{equ:asy-smooth}
\begin{split}
& \sumstar_{(d,2)=1}L(\tfrac{1}{2},\chi_{(1+i)^5d})^4 \Phi\left(\tfrac{N(d)}{X}\right) \\
=& \sumstar_{(d,2)=1}(A_U(d)+B_U(d))^2 \Phi\left(\tfrac{N(d)}{X}\right)\\
=&  \sumstar_{(d,2)=1}A_U(d)^2 \Phi\left(\tfrac{N(d)}{X}\right) + O \Big ( \sumstar_{(d,2)=1}|B_U(d)|^2 \Phi\left(\tfrac{N(d)}{X}\right) + 2\sumstar_{(d,2)=1} |A_U(d)|
|B_U(d)|  \Phi\left(\tfrac{N(d)}{X}\right) \Big ) \\
=&
  \frac{\pi a_{4} \widehat{\Phi}(1)}{2^7 \cdot 3^4 \cdot 5^2 \cdot 7 \cdot \zeta_K(2)} (\frac {\pi}{4})^{10} X (\log X)^{10}+ O \left(X (\log X)^{9.5+\varepsilon}  +  X (\log X)^{-5} \Phi_{(5)} \right),
\end{split}
\end{align}
 where the last expression above follows from an application of the Cauchy-Schwarz inequality to estimate the sum involving the product of $|A_U(d)|$ and
$|B_U(d)|$.

We now take $\Phi(t)$ to be supported on $[1,2]$ satisfying $\Phi(t) =1$ for $t \in (1+\mathcal{Z}^{-1}, 2-\mathcal{Z}^{-1})$ and  $\Phi^{(\nu)} (t)\ll_{\nu} \mathcal{Z}^{\nu}$ for all rational integer $\nu \geq 0$. We then deduce that $\Phi_{(\nu)}  \ll_{\nu} \mathcal{Z}^{\nu}$, and that $\widehat{\Phi}(1) = 1 + O(\mathcal{Z}^{-1})$. Thus \eqref{equ:asy-smooth} implies that
\begin{align*}
 & \sumstar_{(d,2)=1}L(\tfrac{1}{2},\chi_{(1+i)^5d})^4\Phi(\tfrac{N(d)}{X}) \\
 =&
  \frac{\pi a_{4} \widehat{\Phi}(1)}{2^7 \cdot 3^4 \cdot 5^2 \cdot 7 \cdot \zeta_K(2)} (\frac {\pi}{4})^{10}X (\log X)^{10} + O \left(  X (\log X)^{10} \mathcal{Z}^{-1} +X (\log X)^{9.5+\varepsilon}+  X (\log X)^{-5} \mathcal{Z}^5\right).
\end{align*}
  We then deduce by taking $\mathcal{Z}=\log X$ that
\begin{align}
 \label{equ:lower}
 \sumstar_{\substack{(d,2)=1 \\ X < N(d) \leq 2X}}L(\tfrac{1}{2},\chi_{(1+i)^5d})^4
 \geq \sumstar_{(d,2)=1}L(\tfrac{1}{2},\chi_{(1+i)^5d})^4\Phi(\tfrac{N(d)}{X}) =
 \frac{\pi a_{4} \widehat{\Phi}(1)}{2^7 \cdot 3^4 \cdot 5^2 \cdot 7 \cdot \zeta_K(2)} (\frac {\pi}{4})^{10}X (\log X)^{10}  + O \left( X (\log X)^{9.5+\varepsilon} \right).
\end{align}
Similarly,  we can choose  $\Phi(t)$ in (\ref{equ:asy-smooth})  such that $\Phi(t) = 1 $ for $t \in [1 ,  2]$, $\Phi(t) =0$ for all $t \notin (1-\mathcal{Z}^{-1},2+\mathcal{Z}^{-1}) $,  and  $\Phi^{(\nu)}(t) \ll_{\nu} \mathcal{Z}^{\nu}$ for all $\nu \geq 0$. Taking $\mathcal{Z}=\log X$, we can deduce that
\begin{align}
\label{equ:upper}
 \sumstar_{\substack{(d,2)=1 \\ X < N(d) \leq 2X}}L(\tfrac{1}{2},\chi_{(1+i)^5d})^4
 \leq \sumstar_{(d,2)=1}L(\tfrac{1}{2},\chi_{(1+i)^5d})^4\Phi(\tfrac{N(d)}{X})=
  \frac{\pi a_{4} \widehat{\Phi}(1)}{2^7 \cdot 3^4 \cdot 5^2 \cdot 7 \cdot \zeta_K(2)} (\frac {\pi}{4})^{10}X (\log X)^{10}  + O \left( X (\log X)^{9.5+\varepsilon} \right).
\end{align}
Combining \eqref{equ:lower} and \eqref{equ:upper}, we obtain that
\begin{align*}
 \sumstar_{\substack{(d,2)=1 \\ X < N(d) \leq 2X}}L(\tfrac{1}{2},\chi_{(1+i)^5d})^4
=\frac{\pi a_{4} \widehat{\Phi}(1)}{2^7 \cdot 3^4 \cdot 5^2 \cdot 7 \cdot \zeta_K(2)} (\frac {\pi}{4})^{10}X (\log X)^{10}  + O \left( X (\log X)^{9.5+\varepsilon} \right).
\end{align*}
 The assertion of Theorem \ref{main-thm} now follows by summing the above over $X= \frac{x}{2}$, $X= \frac{x}{4}$, $\dots$ and then resetting $x$ to be $X$.

\subsection{Proof of Theorem \ref{theo:mainthm}}
    In this section, we complete the proof for Theorem \ref{theo:mainthm}.  We first apply the Cauchy-Schwartz inequality to see that
\begin{align}
\label{cauchy-lower}
  \sumstar_{(d,2)=1}L (\tfrac{1}{2},\chi_{(1+i)^5d})^4\Phi \left(\tfrac{d}{X} \right)
 \geq \frac {\mathcal{A}^2}{\mathcal{B}},
\end{align}
  where
\begin{align*}
 \mathcal{A}=& \sumstar_{(d,2)=1} A_U(d) L (\tfrac{1}{2},\chi_{(1+i)^5d})^2  \Phi\left(\frac{N(d)}{X} \right), \\
\mathcal{B}=& \sumstar_{(d,2)=1} A_U(d)^2 \Phi\left(\frac{N(d)}{X}\right).
\end{align*}
  Here we recall that $A_U(d)$ is defined as in \eqref{equ:13.1}. 

We can evaluate $\mathcal{B}$ similar to our evaluation of $S(U_1, U_2)$ in Section \ref{sec: pf1stmainthm}, except that we now set $U_1=U_2= U= X^{1-4\varepsilon}$. By setting $Y=X^{\frac{1}{2}}U_1^{-\frac{1}{4}}U_2^{-\frac{1}{4}}$ again, we deduce from \eqref{S1k0est}, Lemma \ref{lem:13.1}, Lemma \ref{lem:k=square} and Lemma \ref{lem:error} that unconditionally, 
\begin{align*}
\mathcal{B}  = \frac{\pi a_{4} \left(1+ O(\varepsilon) \right)}{2^7 \cdot 3^4 \cdot 5^2 \cdot 7 \cdot \zeta_K(2)} (\frac {\pi}{4})^{10} \widehat{\Phi}(1) X (\log X)^{10} + O \left(X (\log X)^{9} + X \Phi_{(5)}\right),
\end{align*}
with the implied constant in $O(\varepsilon)$ being absolute.

To evaluate $\mathcal{A}$, we recast it as
\[
\mathcal{A} = 4 \sumstar_{(d,2)=1}\sum_{\substack{n_1, n_2 \equiv 1 \bmod {(1+i)^3}}} \frac{\chi_{(1+i)^5d}(n_1n_2)d_{[i]}(n_1)d_{[i]}(n_2)}{N(n_1n_2)^{\frac{1}{2}}} h_1(d,n_1,n_2),
\]
where
\[
h_1(x,y,z) = \Phi \left(\frac{N(x)}{X}\right)V\left(\frac{N(y)}{U}\right)V\left(\frac{N(z)}{N(x)}\right).
\]
A similar argument to our evaluation of $\mathcal{B}$ above implies that, by taking $Y=X^{\frac{1}{2}}U^{-\frac{1}{4}} X^{-\frac{1}{4}}$, we have unconditionally,
\begin{align*}
\mathcal{A} =  \frac{\pi a_{4} \left(1+ O(\varepsilon) \right)}{2^7 \cdot 3^4 \cdot 5^2 \cdot 7 \cdot \zeta_K(2)} (\frac {\pi}{4})^{10} \widehat{\Phi}(1) X (\log X)^{10} + O \left(X (\log X)^{9}+X \Phi_{(5)}\right),
\end{align*}
with the implied constant in $O(\varepsilon)$ being absolute.

We now take $\mathcal{Z}=\log X$ and take $\Phi$ so that $\Phi(t) =1$ for $t \in (1+\mathcal{Z}^{-1}  ,  2-\mathcal{Z}^{-1})$, $\Phi(t) =0$ for all $t \notin (1,2) $,  and  $\Phi^{(\nu)} (t)\ll_{\nu} \mathcal{Z}^{\nu}$ for all $\nu \geq 0$. Applying our estimations for $\mathcal{A}$ and $\mathcal{B}$ in \eqref{cauchy-lower}, we deduce that
\begin{align*}
  \sumstar_{\substack{(d,2)=1 \\ X <N(d) \leq 2X}}L (\tfrac{1}{2},\chi_{(1+i)^5d})^4
 \geq  \left(1+ O(\varepsilon) \right)\frac{\pi a_{4}}{2^7 \cdot 3^4 \cdot 5^2 \cdot 7 \cdot \zeta_K(2)} (\frac {\pi}{4})^{10} \widehat{\Phi}(1) X (\log X)^{10}.
\end{align*}
The assertion of Theorem \ref{theo:mainthm} now follows by summing the above over $X= \frac{x}{2}$, $X= \frac{x}{4}$, $\dots$ and then resetting $x$ to be $X$.

\vspace*{.5cm}

\noindent{\bf Acknowledgments.} P. G. is supported in part by NSFC grant 11871082.

\bibliography{biblio}
\bibliographystyle{amsxport}

\vspace*{.5cm}

\end{document}